\documentclass[11pt]{amsart}%12
\usepackage{fullpage}
\usepackage{graphicx}
\usepackage{amsfonts}
\usepackage{amsthm}
\usepackage{amsmath}
\usepackage{amssymb}
\usepackage{hyperref}
\usepackage[arrow,matrix,curve,color]{xy}
\usepackage{pb-diagram,pb-xy}
\usepackage{hyperref}   % must be loaded before enumitem
\usepackage{enumitem}
\usepackage[normalem]{ulem}
\usepackage{color}
\usepackage[normalem]{ulem}
\usepackage{tikz}
\usepackage{xspace}
\definecolor{light-blue}{rgb}{0.8,0.85,1}
\definecolor{light-red}{rgb}{1,.4,.4}
\definecolor{purp}{rgb}{.7,.3,1}
\definecolor{yel}{rgb}{1,1,.5}
\definecolor{cy}{rgb}{0,1,1}
\definecolor{m}{rgb}{1,0.1,1}

\usepackage{colortbl}
\usepackage{multirow}
\usepackage{array}

\theoremstyle{plain}
\newtheorem{theorem}{Theorem}[section]
\newtheorem{corollary}[theorem]{Corollary}
\newtheorem{lemma}[theorem]{Lemma}
\newtheorem{proposition}[theorem]{Proposition}

\theoremstyle{definition}
\newtheorem{remark}[theorem]{Remark}
\newtheorem{definition}[theorem]{Definition}
\newtheorem{example}[theorem]{Example}

\newcommand{\p}{\partial}
\newcommand{\fb}{\mbox{-}\mathrm{fb}}
\newcommand{\fr}{\mathsf{fr}}
\newcommand{\spin}{\mathsf{spin}}

\newcommand{\cyl}{\mathrm{cy}\ell}

\newcommand{\psc}{positive scalar curvature}
\newcommand{\co}{\colon\,}

\newcommand{\bC}{\mathbb C}

\newcommand{\bF}{\mathbb F}
\newcommand{\bH}{\mathbb H}

\newcommand{\bZ}{\mathbb Z}

\newcommand{\bP}{\mathbb P}
\newcommand{\bN}{\mathbb N}
\newcommand{\bQ}{\mathbb Q}

\newcommand{\cL}{\mathcal L}

\newcommand{\fp}{\mathfrak p}
\newcommand{\fg}{\mathfrak g}
\newcommand{\fk}{\mathfrak k}
\newcommand{\fN}{\mathfrak N}

\newcommand{\SU}{\mathop{\rm SU}}

\newcommand{\Sp}{\mathop{\rm Sp}}
\newcommand{\Spin}{\mathop{\rm Spin}}

\newcommand{\lp}{\textup{(}}
\newcommand{\rp}{\textup{)}}

\newcommand{\Diff}{\operatorname{Diff}}
\newcommand{\BDiff}{\operatorname{BDiff}}

\def\Di{\mathfrak{D}\kern-6.5pt/}
\def\Spi{\mathfrak{S}\kern-6.5pt/}

\title[Psc on simply connected pseudomanifolds]{Positive scalar
  curvature on simply connected\\ spin pseudomanifolds}
\author{Boris Botvinnik} \address{Department of
  Mathematics\\ University of Oregon\\ Eugene OR 97403-1222, USA}
\email[Boris Botvinnik]{botvinn@uoregon.edu}
\urladdr{http://pages.uoregon.edu/botvinn/} \author{Paolo Piazza}
\address{Dipartimento di Matematica\\ Universit\`a di Roma La
  Sapienza\\ Piazzale Aldo Moro\\ 00185 Roma, Italy} \email[Paolo
  Piazza]{piazza@mat.uniroma1.it}
\urladdr{http://www1.mat.uniroma1.it/people/piazza/} \author{Jonathan
  Rosenberg} \address{Department of Mathematics\\ University of
  Maryland\\ College Park, MD 20742-4015, USA} \email[Jonathan
  Rosenberg]{jmr@math.umd.edu}
\urladdr{http://www2.math.umd.edu/\raisebox{-3pt}{~}jmr/}

\begin{document}
\begin{abstract}
Let $M_\Sigma$ be an $n$-dimensional Thom-Mather stratified space of
depth $1$.  We denote by $\beta M$ the singular locus and by $L$ the
associated link.  In this paper we study the problem of when such a
space can be endowed with a wedge metric of positive scalar
curvature. We relate this problem to recent work on index theory on
stratified spaces, giving first an obstruction to the existence of
such a metric in terms of a wedge $\alpha$-class $\alpha_w
(M_\Sigma)\in KO_n$. In order to establish a sufficient condition we
need to assume additional structure: we assume that the link of
$M_\Sigma$ is a homogeneous space of positive scalar curvature,
$L=G/K$, where the semisimple compact Lie group $G$ acts transitively
on $L$ by isometries. Examples of such manifolds include compact
semisimple Lie groups and Riemannian symmetric spaces of compact
type. Under these assumptions, when $M_\Sigma$ and $\beta M$ are spin,
we reinterpret our obstruction in terms of two $\alpha$-classes
associated to the resolution of $M_\Sigma$, $M$, and to the singular
locus $\beta M$. Finally, when $M_\Sigma$, $\beta M$, $L$, and $G$ are
simply connected and $\dim M$ is big enough, and when some other
conditions on $L$ (satisfied in a large number of cases) hold, we
establish the main result of this article, showing that the vanishing
of these two $\alpha$-classes is also sufficient for the existence of
a well-adapted wedge metric of positive scalar curvature.
\end{abstract}
\keywords{positive scalar curvature, pseudomanifold, singularity, bordism, transfer, $K$-theory, index}
\subjclass[2010]{Primary 53C21; Secondary 58J22, 53C27, 19L41, 55N22}

\maketitle

\vspace*{-5mm}

\section{Introduction}
\label{sec:intro}

This paper continues a program begun in \cite{MR1857524} and in
\cite{BR}, to understand obstructions to positive scalar curvature
(which we will sometime abbreviate as psc) on manifolds with fibered
singularities, for metrics that are well adapted to the singularity
structure.

In the cases studied in this paper, the stratified spaces or singular
manifolds $M_\Sigma$ that we study are Thom-Mather pseudomanifolds of
depth one.  For the existence theorem we shall take the two strata to
be spin and simply connected; more general situations, involving
non-trivial fundamental groups, will be dealt with in a forthcoming
article \cite{BB-PP-JR}.  Topologically, $M_\Sigma$ is
homeomorphic to a quotient space of a compact smooth manifold $M$ with
fibered boundary $\partial M$.  Then $M$ is called the
\emph{resolution} of $M_\Sigma$, and the quotient map $M\to M_\Sigma$
is the identity on the interior $\mathring M$ of $M$, and on $\partial
M$, collapses the fibers of a fiber bundle $\varphi\co\partial M\to
\beta M$, with fibers all diffeomorphic to a fixed manifold $L$,
called the \emph{link} of the singularity, and with base $\beta M$
sometimes called the Bockstein (by analogy with other cases in
topology).  We briefly refer to these spaces as \emph{manifolds with
  $L$-fibered singularities}.

Note that the structure group of the bundle $\varphi$ can be an
arbitrary subgroup of $\Diff(L)$, and for part of our results we do
consider this general situation.  However, in studying the existence
problem for wedge metrics of positive scalar curvature we shall need
to have more structure in order to relate the topology of the bundle,
in particular its bordism theory, with its differential geometric
features.

To this end we assume that the fiber bundle $\varphi\co\partial M\to
\beta M$ comes from a principal $G$-bundle $p\co P\to \beta M$, for
some connected compact Lie group $G$ that acts transitively on $L$ by
isometries for some fixed metric $g_L$, and thus $\partial M =
P\times_G L$.  The transitivity of the action of $G$ on $L$ means that
$L=G/K$ is a homogeneous space and has constant scalar curvature. We
refer to these special singular spaces as {\em manifolds with
  $(L,G)$-fibered singularities}.

Since the space $M_\Sigma$ is not a smooth manifold in general (it
will be if and only if $L$ is a standard sphere), we need to explain
what kind of metric we want to use. We shall employ \emph{wedge
  metrics}, also called \emph{iterated conic metrics}, on the regular
part of $M_\Sigma$. These are defined as follows. First, we identify
$M_\Sigma$ with a union $M_\Sigma = M\cup_{\p M} - N(\beta M)$, where
$M$ is a manifold with fibered boundary and $N(\beta M)$ is a tubular
neighborhood of the singular locus $\beta M$.  Then on $M$ we consider
a metric $g_M$ which is a product metric $dt^2 + g_{\partial M}$ in a
small collar neighborhood $\partial M\times [0, \varepsilon)$ of the
  boundary $\partial M$; we assume that $g_{\partial M}$ is a
  submersion metric for the bundle $\varphi\co\partial M\to \beta M$,
  with horizontal metric $\phi^* g_{\beta M}$ and vertical metric
  $g_{\partial M/\beta M}$. On the regular part of $N(\beta M)$ we
  consider a metric of the following type:
\begin{equation*}
g_{N (\beta M)} = dr^2 + r^2 g_{\partial M/\beta M} + \varphi^*
g_{\beta M} + O(r)\,.
\end{equation*}
We call such metrics on the regular part of
$M_\Sigma$ \emph{adapted wedge metrics}. Notice that we can consider
adapted wedge metrics even if the link is $S^n$; in that case we are
considering special metrics on a smooth ambient manifold $M_\Sigma$
with respect to a specified submanifold $\beta M$.

In the case in which $L$ is a homogeneous space, as above, there is a
natural submersion metric on $\partial M\xrightarrow{\varphi} \beta M$
which is defined as follows. A connection $\nabla^p$ on the principal
bundle $p\co P\to \beta M$ gives a connection $\nabla^\varphi$ on the
associated fiber bundle $\varphi\co \partial M\to \beta M$. Then,
since the structure group $G$ of the bundle $\varphi\co \partial M\to
\beta M$ acts by isometries of the metric $g_L$, the connection
$\nabla^\varphi$ provides an orthogonal splitting of the metric $g_{\p
  M}$ into the \emph{horizontal} metric lifted from $g_{\beta M}$ and
the \emph{vertical metric} $g_L$.  For the details, see Section
\ref{sec:setup}.  In other words, when the link $L$ is a homogeneous
space, the wedge metric near the singularity is determined by a metric
$g_{\beta M}$ on the singular locus $\beta M$ and the standard metric
$g_L$ on $L$, together with the connection.  To make the problem of
existence of an adapted metric of {\psc} maximally interesting, we add
one more condition, that $L$ have constant positive scalar curvature $\kappa_L$
precisely equal to the scalar curvature $\kappa_\ell$ of the standard
round $\ell$-sphere, where $\ell=\dim L$.
This insures that the cones over
$L$ are actually scalar-flat; see Section \ref{sec:setup}. An adapted
wedge metric $g$ on $M_{\Sigma}$ (with no $O(r)$ error term)
which satisfies those additional
conditions is called a \emph{well-adapted wedge metric}.

We note here that since the metric $g_L$ has positive
scalar curvature, it follows that the metric $g_{\p M}$ on the
boundary $\p M$ could be assumed to have positive scalar curvature as
well. In the case when all manifolds are spin, this implies that the
corresponding Dirac operator on $M$ gives an
invariant $\alpha_{{\rm cyl}} (M)\in KO_n$, defined by
attaching a cylindrical end $\p M \times [0,\infty)$ to $M$
and giving it the product metric  $g_{\p M} + dt^2$. Then we take
the $KO_n$-valued index of the $C\ell_n$-linear Dirac operator
\cite[\S II.7]{lawson89:_spin} on the resulting
noncompact manifold $M_\infty:=M\cup_{\p M}(\p M \times [0,\infty))$;
the operator is Fredholm since the scalar curvature
is bounded away from $0$ except on a compact set.

In addition, the Dirac operator on the
manifold $\beta M$ also determines a corresponding class $\alpha(\beta
M)\in KO_{n-1-\ell}$. For later use, we use the notation
$\kappa_\ell=\ell(\ell-1)$ for the scalar curvature of the unit
sphere $S^{\ell}(1)$. Now we can state our main results on manifolds with
$(L,G)$-singularities.
\begin{theorem}[Obstruction Theorem]
\label{thm:psobstruction} 
Let $L=G/K$ be a homogeneous space, $\dim L = \ell$, where $G$ is a
connected compact semisimple Lie group, and $g_L$ be a $G$-invariant
Riemannian metric on $L$ of constant scalar curvature equal to
$\kappa_\ell=\ell(\ell-1)$.  Let $M_\Sigma = M\cup_{\p M} - N(\beta
M)$, where $M$ and $\beta M$ are compact spin manifolds, and the
boundary $\p M = P\times_G L$ for some principal $G$-bundle $p\co P\to
\beta M$ with classifying map $\beta M\to BG$.
Assume that $M_\Sigma$ admits a well-adapted metric of {\psc}. Then
the $\alpha$-invariants $\alpha_{{\rm cyl}} (M)\in KO_n$ and
$\alpha(\beta M) \in KO_{n-\ell-1}$ both vanish.
\end{theorem}
Some of the background leading to the Obstruction Theorem will be
discussed in Section \ref{sect:analysis}. There,
building on work of Albin-Gell-Redman,
we shall in fact introduce a wedge alpha-class
$\alpha_w (M_\Sigma, g)\in KO_n$
which is defined under much weaker conditions and
that yields the most general Obstruction Theorem, Theorem
\ref{theo:geometric-witt-part1}, for simply connected manifolds with
$L$-fibered singularities. The special case treated in Theorem
\ref{thm:psobstruction}, namely manifolds with $(L,G)$-fibered
singularities, is treated later, in Section \ref{sec:obstructions}. We
have singled out this version of the Obstruction Theorem now because
it is this theorem for which we will prove a converse, namely an
existence result under the assumption that the two obstructions
vanish.

To state the existence result, we need one more definition. Let $X$ be
a closed spin manifold endowed with a $G$-action. Let $g_X$ be a
psc-metric on $X$. We say that a $(X,g_X)$ is a \emph{spin
  psc-$G$-boundary} if there exists a spin $G$-manifold $Z$ bounding
$X$ as a spin $G$-manifold and a psc-metric $g_Z$ on $Z$ which is a
product metric near the boundary with $g_{Z}|_X=g_{X}$.
\begin{theorem}[Existence Theorem]
\label{thm:pscsufficiency} 
Let $L=G/K$ be a homogeneous space, $\dim L = \ell$,
where $G$ is a connected compact semisimple Lie group, and $g_L$ be a
$G$-invariant Riemannian metric on $L$ of constant scalar curvature
equal to $\kappa_\ell=\ell(\ell-1)$.  Let $M_\Sigma = M\cup_{\p M}
- N(\beta M)$, where $M$ and $\beta M$ are compact spin manifolds, and
the boundary $\p M = P\times_G L$ for some principal $G$-bundle 
$p\co P\to \beta M$ with classifying map $\beta M\to BG$.
Assume $M$, $\beta M$, $L$, and $G$ are simply
connected and $n\ge \ell+6$.

Then $M_{\Sigma}$ admits a well-adapted psc-metric if and only if the
$\alpha$-invariants $\alpha_{{\rm cyl}} (M)\in KO_n$ and $\alpha(\beta
M) \in KO_{n-\ell-1}$ both vanish, provided one of the following
conditions holds:
\begin{enumerate}
\item[{\rm (i)}] the manifold $(L,g_L)$ is a spin
  psc-$G$-boundary; \textbf{or}
\item[{\rm (ii)}] the bordism class $[\beta M\to BG]$
  vanishes in $\Omega_{n-\ell - 1}^\spin(BG)$. 
\end{enumerate}
\end{theorem}
%\newpage
%\vspace*{2mm}
\begin{remark}
  \leavevmode
\begin{enumerate}
\item The dimensional assumption $n\ge \ell+6$ is necessary
  in order to apply surgery theory to $M$ and to $\beta M$.
\item The assumption (i) holds when $L$ is a sphere,
  an odd complex projective space, or when $L=G$. Other cases
  where this assumption holds are discussed in Remark
  \ref{rem:morecasesGbdy}.
\item The assumption (ii) holds automatically if 
  $\p M=\beta M\times L$ (i.e., the bundle $\varphi$ is trivial) and if
  $L$ is an even quaternionic projective space.
\end{enumerate}  
\end{remark}
Sections \ref{sect:analysis}, \ref{sec:setup}, and \ref{sec:main}
contain important preliminaries needed for the proofs of these
theorems. The proof of Theorem \ref{thm:psobstruction} is completed in
Section \ref{sec:obstructions}, and the proof of Theorem
\ref{thm:pscsufficiency} is in Section \ref{sec:existence}.

Extensions of these results to the case where $M$ and $\p M$ are not
necessarily simply connected will be found in our sequel paper
\cite{BB-PP-JR}.  A quick sketch of the contents of that paper is in
Section \ref{sec:preview}. 

\medskip
\noindent
{\bf Acknowledgments.} This research was partially supported by {U.S.}
NSF grant number DMS-1607162
and by Universit\`a di Roma La Sapienza. {B.B.} and {J.R.}
acknowledge a very pleasant visit to Rome in May--June 2019, during
which time much of this work was done. {B.B.} was also partially supported
by Simons Foundation Collaboration Grant number 708183.  We would also like
the referee for his/her careful reading of the first draft and for suggestions
for improvements.

\section{Dirac operators and associated $KO$ classes}
\label{sect:analysis}
\subsection{Introduction}
In this section we review and revisit necessary constructions and
results concerning Dirac operators on smooth spin pseudomanifolds with
depth-1 singularities. Our goal here is to describe under what
conditions a Dirac operator defines a corresponding $KO$-homology
class.  We shall proceed in some generality first and then specify
further assumptions on the link fibration and get sharper results
correspondingly.

A particular case of pseudomanifolds with depth-1 singularities is
given by manifolds with Baas-Sullivan singularities \cite{MR0346824},
when a type of singularity (a link $L$) is fixed. 
Starting with a smooth manifold $M$ with boundary $\p
M=\beta M \times L$, we obtain a manifold with $L$-singularity
$M_{\Sigma}$ by gluing $M$ to the product $\beta M\times c(L)$ (where
$c(L)$ is a cone over $L$) along $\p M$, i.e., $M_{\Sigma}=M\cup
-\beta M\times c(L)$. It will be convenient to adapt similar notations
for pseudomanifolds with depth-$1$ singularities.

\subsection{Pseudomanifolds of depth 1}
First we recall necessary definitions. For the rest of this section,
we fix a closed smooth compact manifold $L$,
called the \emph{link}.
In terms of the notation of \cite{MR1857524}, the singularity
type of our pseudomanifolds is $\Sigma=(L)$. 
We consider a depth-1 Thom-Mather pseudomanifold
$M_\Sigma$ with dense stratum $M_\Sigma^{{\rm reg}}$, singular stratum
$\beta M$, and associated link $L$.  We shall not go 
into a detailed explanation of the axioms of a Thom-Mather space
--- see, for example, \cite[Section A.1]{Brasselet-et-al}
and \cite[Definition 2.1]{ALMP-novikov} --- but we simply recall
that as a consequence of these axioms,
we have a locally compact metrizable space $M_\Sigma$ such that:
\begin{itemize}
\item the space $M_\Sigma$ is the union of two
  smooth strata, $M^{{\rm reg}}_\Sigma$ and $\beta M$;
\item  the manifold $M^{{\rm reg}}_\Sigma$ is open and dense in $M_\Sigma$;
\item the manifold $\beta M$ is smooth and compact;
\item there is an open neighborhood $N(\beta M)$ of $\beta M$ in
  $M_\Sigma$, equipped with a continuous retraction
  $\mathsf{re}\co N(\beta M)\to   \beta M$
  and a continuous map $\rho\colon N(\beta M) \to   [0,+\infty)$ 
  which is smooth in $N(\beta M)^{{\rm reg}}:=N(\beta M)\cap M^{{\rm reg}}_\Sigma$
  and  such that $\rho^{-1}(0)= \beta M$;
\item the neighborhood $N(\beta M)$ is a (locally trivial) fiber bundle 
  over $\beta M$  (via the above retraction $\mathsf{re}$)
  with fiber $c(L)$, the cone over $L$, and with radial 
  variable along the cones induced by $\rho$.
\end{itemize}
We can associate to $M_\Sigma$ its \emph{resolution}, which in our
case is a manifold $M:=M_\Sigma\setminus \rho^{-1}([0,1))$ with
boundary $\partial M:=\rho^{-1}(1)$, such that $\p M$ is the total
space of a smooth fiber bundle
$L \to\partial M \xrightarrow{\varphi} \beta M$
with fiber $L$. This fibration is sometimes called the
\emph{link bundle} associated to $M_\Sigma$.
Clearly, there is a diffeomorphism between the
interior $\mathring{M}$ of $M$ and $M_\Sigma^{{\rm reg}}$.  Once a
link $\Sigma=(L)$ is fixed, we will call such a pseudomanifold
$M_\Sigma$ a \emph{pseudomanifold with fibered $L$-singularity}.
Sometimes we abuse terminology by calling the resolution $M$ of
$M_\Sigma$ a \emph{manifold with fibered $L$-singularity}, in cases
where the extra structure is understood.

Thus we have $M_{\Sigma}:= M\cup_{\p M} - N(\beta M)$,
where $N(\beta M)$ comes together with a fiber-bundle
\begin{equation*}
  c(L) \to N (\beta M) \xrightarrow{\varphi_c} \beta M .
\end{equation*}
(We have used the symbol $\varphi_c$ here since this
bundle is just the result of replacing each fiber $L$ in
the fibration $\varphi$ by the cone $c(L)$ over $L$.)
Moreover if $v$ is the vertex of $c({L})$, then the inclusion
$\{v\}\hookrightarrow {c(L)}$ induces an embedding:
\begin{equation*}  
\beta M \hookrightarrow N(\beta  M) \subset M_{\Sigma}.
\end{equation*}
\subsection{Wedge metrics}
Let $M_\Sigma$ be a depth-1 pseudomanifold as above.  A Riemannian
metric on $M_\Sigma$ is, by definition, a Riemannian metric $g$ on
$M_\Sigma^{{\rm reg}}$.  We shall consider special types of
metrics. To this end we fix $g_{\p M}$, a Riemannian metric on $\p M$
and $g_{\beta M}$, a Riemannian metric on $\beta M$. We assume that
$\varphi\co \p M \to \beta M$ is a Riemannian submersion; this means
that we have fixed a connection on $\p M$, that is a splitting $T(\p
M)= T(\p M/\beta M) \oplus T_H ( \p M)$ with $T_H (\p M)\simeq
\varphi^* T(\beta M)$ and $g_{\p M}= h_{\p M /\beta M} \oplus^{\perp}
\varphi^*g_{\beta M}$,\footnote{The symbol $\oplus^{\perp}$ means
  orthogonal direct sum, where the pulled-back metric from the
  base $\beta M$ is put on the horizontal subspaces as defined by
  the connection.}
with $h_{\p M /\beta M} $ a metric on the
vertical tangent bundle $T(\p M/\beta M)$ of the fibration
$L\to \p M\xrightarrow{\varphi}\beta M$.
Let $r$ now be the radial variable along
the cones with $r=0$ corresponding to $\beta M$.
\begin{definition}\label{admissible}
We say that $g$ on $ M_\Sigma^{{\rm reg}} $ is a \emph{wedge metric}
if on $N(\beta M)$ it can be written as
\begin{equation}\label{conic-metric}
dr^2+r^2h_{\p M/\beta M}+ \varphi_c^*g_{\beta M} + O(r)\,.
\end{equation}
Equation \eqref{conic-metric} says that
the difference between $g$ and a metric of the form
$dr^2+r^2h_{\p M/\beta M}+ \varphi_c^*g_{\beta M}$ is a smooth
section of the symmetric tensor product of the cotangent bundle
vanishing at $r=0$.
If, in addition, $g$ is of product type near $\p N(\beta M)=\p M$,
then we call $g$ an \emph{adapted} wedge metric.  We refer to the pair
$(M_\Sigma,g)$ as a \emph{wedge space} (of depth one).
\end{definition}
Notice that $g$ is an incomplete Riemannian metric on the open
manifold $M_\Sigma^{{\rm reg}}$.  Using the diffeomorphism between the
interior $\mathring{M}\subset M$ of the resolved manifold and
$M_\Sigma^{\rm reg}$, we can induce a metric on $\mathring{M}$,
denoted again $g$.  We redefine $r$ to be the boundary defining
function for $\p M\subset M$.  In connection with a wedge metric $g$,
we can consider the \emph{wedge tangent bundle} ${}^w TM$ over the resolved
manifold $M$ (also called the \emph{incomplete edge tangent bundle}). See 
for example  \cite[Section 2.1]{albin-gellredman-sigma}.
The wedge tangent bundle  is dual %(with respect to the wedge metric) 
to the wedge cotangent bundle ${}^w T^*M$, 
defined through the Serre-Swan theorem \cite{MR143225}
as the vector bundle whose sections are given
by the finitely generated projective $C^\infty (M)$-module
$\{\omega\in C^\infty(M,T^*M)\;\;|\;\; \iota_\partial ^* \omega \in \varphi^* C^\infty (\beta M, T^* (\beta M))\}$,
with $\iota_\partial$ denoting the inclusion of $\partial M$ into $M$.
As explained in detail in \cite[Section 2.1]{albin-gellredman-sigma},
  the wedge cotangent bundle has a local basis given by the 1-forms
$\{dr,\ rd\lambda_1,\dots , rd\lambda_k,\ dy_1,\dots, dy_\ell\}$ with
$\lambda_1,\dots,\lambda_{\ell}$ local coordinates on $L$, where
$\ell=\dim L$, and $y_1,\dots,y_k$ local coordinates on $\beta M$;
notice that the differential forms $\{rd\lambda_1,\dots , rd\lambda_k\}$ vanish 
on the boundary as sections of the cotangent bundle $T^*M$
but they do not vanish as sections of ${}^w T^*M$.
See \cite[Section 2.1]{albin-gellredman-sigma}.
It is easy to show that a wedge metric $g$ extends from the interior
of $M$ to a smooth
metric on ${}^w TM\to M$. We remark  that $M$ can also be given the collared
metric $g|_{M}$; we denote this metric by $g_M$; needless to say, $g$ and $g_M$ 
have different behavior near the boundary.

We also have the \emph{edge tangent bundle}, ${}^e TM\to M$, 
see \cite{MR1133743},
defined  through the Serre-Swann theorem as the vector bundle 
whose sections are given by the finitely generated projective
$C^\infty (M)$-module
\begin{equation}\label{edge-v-f}
  \mathcal{V}_{{\rm e}}=\{V\in C^\infty (M,TM)\;\;
  \text{such that}\;\; V |_{\p M} \;\;
  \text{is tangent to the fibers of} \;\;\p M
  \xrightarrow{\varphi} \beta M\}
\end{equation}  
The vector fields in  $\mathcal{V}_{{\rm e}}$ are called
\emph{edge vector fields}.
The edge tangent bundle ${}^e TM\to M$ is  locally
spanned by the vector fields
\begin{equation}\label{edge-vf}
  r\partial_r\,,\quad  \partial_{\lambda_1},\dots ,
  \partial_{\lambda_{\ell}}\,,\quad r\partial_{y_1},\dots, r\partial_{y_k}.
\end{equation}
where, as before, the vector fields
$\{ \partial_{\lambda_{\ell}}\,,\quad r\partial_{y_1},\dots, r\partial_{y_k}\}$
vanish on the boundary as sections of the tangent bundle $TM$
but they are non-vanishing as sections of ${}^e TM\to M$.
See \cite{Melrose} for more on this general philosophy. 
What is interesting about $\mathcal{V}_{{\rm e}}$ is
that it is a Lie algebra; this means that it generates an algebra of
differential operators known as \emph{the algebra of edge differential
  operators}, denoted ${\rm Diff}^*_e (M)$. For edge differential
operator there is a pseudodifferential calculus developed by Mazzeo
\cite{MR1133743} and this calculus plays a central role in the
analysis of Dirac-type operators associated to wedge metrics.

We say that a wedge metric $g$ is a psc-metric if it has positive
scalar curvature as a Riemannian metric on $ M_\Sigma^{{\rm reg}}$.
In this paper we are interested in understanding necessary and (for
some particular links $L$) sufficient conditions when a pseudomanifold
$M_\Sigma$ with $L$-fibered singularity admits a psc wedge
metric. When this space of psc-metrics is non-empty, we shall
investigate in the sequel paper \cite{BB-PP-JR} its topological
properties, for example the number of its connected components or the
cardinality of its homotopy groups.

Note that stratified pseudomanifolds can also be endowed with
different metrics, for example edge metrics or $\Phi$-metrics (in
contrast with the wedge case, these are \emph{complete} metrics).  For
a rather detailed study of the resulting index theoretic obstructions
to the existence of $\Phi$-metrics with psc, see \cite{PiZen}.

\subsection{Spin-stratified pseudomanifolds and
  Dirac operators: $KO$-homology classes}
Let $(M_\Sigma,g)$ be a wedge space of depth 1,
with singular locus $\beta M$  and link $L$. We assume
that the resolved manifold $M$ admits a spin structure
and we fix such a structure. (It would be equivalent to
fix a spin structure on ${}^w TM$, since this is isomorphic
to $TM$ as bundles.) This fixes a spin structure on $\p M$
too. We assume additionally that 
$\beta M$ is spin and fix a spin structure for $\beta M$.
Notice that this also fixes a spin structure
for the vertical tangent bundle of the boundary fibration, 
denoted $T(\partial M/\beta M)$, which is
endowed with the vertical metric $g_{\p M/\beta M}$.
See \cite[II, Prop. 1.15]{lawson89:_spin}.

\begin{definition}
Let $(M_\Sigma,g)$ be a wedge space of depth 1, singular locus $\beta M$
and link $L$. We shall say that $M_\Sigma$ is \textbf{spin-stratified}
if both $M$ and $\beta M$ are spin.
\end{definition}

We  denote by $\Spi_g (M)\to M$ the bundle given by
$P_{{\rm spin}}\times_{\lambda} C\ell_n$, with $\lambda\co {\rm Spin}_n\to
{\rm Hom} (C\ell_n,C\ell_n)$ the representation given by left
multiplication. See \cite[Chapter II, Section 7]{lawson89:_spin}. 
There is a fiberwise right action of $C\ell_n$ on 
$\Spi_g (M)$ on the right which makes $\Spi_g (M)$ a bundle of
rank 1 $C\ell_n$-modules. Notice that
$\Spi_g (M)$ is graded:  $\Spi_g (M)=\Spi_g^0 (M)\oplus \Spi_g^1 (M)$.
Let $\Di_g$ be the
associated $C\ell_n$-linear Atiyah-Singer operator. The operator
$\Di_g$ is defined in the usual fashion, $\Di_g:= {\rm cl}\circ \nabla$, with
$\nabla$  the Levi-Civita connection on 
 $\Spi_g (M)$; it has the local expression 
$$\Di_g :=\sum_j {\rm cl}(e^j)\nabla_{e_j}$$
with $\{e_j\}$ a local orthonormal frame  of vector fields
and $\{e^j\}$ the dual basis defined by the metric.
See again \cite[Chapter II, Section 7]{lawson89:_spin} for the details.
$\Di_g$ is a $\bZ/2$-graded odd  formally self-adjoint operator of
Dirac type commuting with the right action of $C\ell_n$.
For this operator the Schr\"odinger-Lichnerowicz formula holds:
\begin{equation}\label{lichn}
\Di^2=\nabla^*\nabla+ \frac{1}{4}\kappa_g,
\end{equation}
with $\kappa_g$ denoting the scalar curvature of $g$.
See again \cite[Chapter II, Section 7]{lawson89:_spin} and also
\cite{Rosenberg-III} for more details on this crucial point.

\medskip
\noindent
\textbf{Notation.} Unless absolutely necessary we shall omit the reference
to the wedge metric $g$ in the bundle and the operator, thus denoting the
$C\ell_n$-linear Atiyah-Singer operator  simply by $\Di$.
Moreover, we shall often  use the shorter notation $\Spi$
instead of $\Spi (M)$.

\medskip
We can regard $\Di$ as a wedge differential operator of order 1 on $L^2
(M,\Spi)$---more on this in a moment---initially with domain equal to
$C^\infty_c (M_\Sigma^{{\rm reg}} ,\Spi)\equiv C^\infty_c (\mathring{M},\Spi)
\subset L^2 (M,\Spi)$.  We are looking for self-adjoint $C\ell_n$-linear 
extensions of this differential operator in $L^2 (M,\Spi)$.
Note, crucially, that the analysis given
in \cite{albin-gellredman-sigma} and \cite{a-gr-dirac} is quite
general, since it applies to any Dirac-type operator.

Associated to $\Di$ there is a well
defined boundary operator  $\Di_{\partial}$, 
which is nothing but the Atiyah-Singer operator of the spin manifold
$\partial M$.
We are assuming that $\partial M$ is a fiber bundle of spin manifolds,
$L\to \p M\xrightarrow{\varphi} \beta M$; consequently
\[
\mathfrak{S} (\p M)\simeq
\mathfrak{S}( \p M/\beta M)\hat{\otimes}\varphi^* \mathfrak{S}(\beta M).
\]
The careful study made by Albin and Gell-Redman of the Levi-Civita
connection near the singular stratum, see Section 2.2 and also
Section 3.1 there, implies that
\begin{equation}\label{AS-operator-near-stratum}
\Di= {\rm cl}(dr) \left( \partial_r + \frac{\ell}{2r} + \frac{1}{r}\Di_{\partial M/\beta M}\hat{\otimes} {\rm Id} + 
 {\rm Id}\hat{\otimes}\widetilde{\Di}_{\beta M} \right) + \mathfrak{B}
\end{equation}
with $\ell=\dim L$, 
$\Di_{\p M/\beta M}$ the vertical family of Atiyah-Singer operators on 
the fibration $L\to \p M\xrightarrow{\varphi} \beta M$,
$\widetilde{\Di}_{\beta M}$ an explicit horizontal operator, and
$\mathfrak{B}$ a bundle endomorphism 
which is $O(r)$. Formula \eqref{AS-operator-near-stratum}
exhibits $\Di$ as a wedge differential operator of degree 1.
With a small abuse of notation, widely used in family index theory,
we denote by $\Di_{L}$ the generic operator of this vertical family.
The following result holds:

\begin{theorem}\label{theo:geometric-witt-part1}
Let $M_\Sigma$ be spin-stratified and assume that  
\begin{equation}\label{geometric-witt-part2}
  {\rm spec}_{L^2} (\Di_{L})\cap (-1/2,1/2)=\emptyset\quad
  \text{for each fiber } L.
\end{equation} Then:
\begin{enumerate}
\item[{\rm 1]}] The operator $\Di$ with domain
$C^\infty_c (M_\Sigma^{{\rm reg}},\Spi)\subset L^2 (M,\Spi)$ is
essentially self-adjoint.
\item[{\rm 2]}] %  \item{2]}
        Its unique self-adjoint extension, still denoted by $\Di$,
        defines  a $C\ell_n$-linear 
        Fredholm operator and thus a class
        $\alpha_w (M_\Sigma,g)$ in $KO_n$, with $n=\dim M_\Sigma$.
\end{enumerate}        

If  the vertical metric $g_{\p M/\beta M}$ is of psc along the fibers,
then there exists $\epsilon > 0$ such that
\begin{equation*}
  {\rm spec}_{L^2} (\Di_{L})\cap (-\epsilon,\epsilon)
  =\emptyset\quad \text{for each fiber } L.
\end{equation*} 
Thus, by suitably rescaling the wedge metric $g$
along the vertical direction of the boundary fibration, we can achieve
condition  \eqref{geometric-witt-part2}
by assuming only that the vertical metric $g_{\p M/\beta M}$ is of
psc along the fibers.\footnote{This is analogous
to the case of the signature operator
on Witt and Cheeger spaces; see \cite{almp-package,ALMP-cheeger}.}

Finally, if $(M_\Sigma^{{\rm reg}},g)$ has psc everywhere then
$\Di$ is $L^2$-invertible;
 in particular  $\alpha_w (M_\Sigma,g)=0$ in $KO_n$.
\end{theorem}

\begin{proof}
Items 1] and 2] follow directly from the microlocal analysis methods 
employed by  Albin and Gell-Redman, 
and more specifically from the explicit form of the normal operator
associated to $\Di$
that one obtains from \eqref{AS-operator-near-stratum}.
See  \cite{albin-gellredman-sigma} and \cite{a-gr-dirac}.
The result relative to the psc assumption along the fibers follows 
as in Bismut-Cheeger \cite[Section 4]{BC-adiabatic}, using 
the Schr\"odinger-Lichnerowicz formula along the fibers. The last
statement follows also 
from the Schr\"odinger-Lichnerowicz formula.
\end{proof}

\begin{definition}\label{def:psc-Witt}
We shall say that the spin-stratified  pseudomanifold
$(M_\Sigma,g)$ is \textbf{psc-Witt} if the metric 
$g$ is of psc along the links, i.e. if the vertical metric
$g_{\p M/\beta M}$ induces on each fiber $L$ a metric of psc.
\end{definition}

\begin{remark}
In this work we have limited
ourselves to the $\alpha$-class in $KO_n$. 
$\alpha$-classes in $KO$ of suitable real $C^*$-algebras will be
treated in the sequel \cite{BB-PP-JR} of this article. 
\end{remark}

\subsection{Dependence on a metric}
\label{remark:dependence}
Contrary to the case of closed manifolds, the wedge
$\alpha$-class $\alpha_w(M_\Sigma,g)$ of a psc-Witt pseudomanifold
can depend on the choice of the
adapted wedge metric $g$ near the singular locus $\beta M$.
On the other hand, if $g(t)$ is a smooth 1-parameter family of adapted
wedge metrics that are psc-Witt for each $t$, then $\alpha_w (M_\Sigma,g_t)$
in $KO_n$ does not depend on $t$. Indeed,
since our assumption ensures that condition
\eqref{geometric-witt-part2} is satisfied for all
$t$ (up to scaling), it is possible to construct a continuous
family of parametrices.
The result then follows from
\cite[Ch. III,  Theorem 10.8]{lawson89:_spin}.

\subsection{Cylindrical KO-theory classes and a gluing formula}
We decompose 
\[
M_\Sigma= M\cup_{\p M} (-N(\beta M))\quad\text{and}\quad
M_\Sigma^{{\rm reg}}= M\cup_{\p M} (-N(\beta M)^{{\rm reg}}).
\]
Let $g$ be an adapted wedge metric on $M_\Sigma$. Denote by $g_M$ the
Riemannian metric induced by restriction of $g$ to $M$;
similarly, let $g_{N(\beta M)}$ be  the
metric induced by restriction of $g$ to $N(\beta M)$,
the collar neighborhood of the singular stratum. By assumption
$$ g_{N(\beta M)}=dr^2 + r^2 g_{\p M/\beta M} + \varphi_c^* g_{\beta M}+ O(r)\,.$$ 
Recall that an adapted wedge metric $g$ is such that 
$g_M$  and $g_{N(\beta M)}$  are of product type  in a collar
neighborhood of  $\p M$. We make the hypothesis  that 
not only does $g_{\p M/\beta M}$ have psc along the fibers but that
the whole metric $g_{\p M}$ is of psc. 
As explained below, a sufficient condition for this additional
property to hold  is that the fibers are totally geodesic in $\p M$.
Now attach an infinite cylinder to $M$ along the boundary $\p M$ and
extend the metric to be constant on the cylinder, obtaining $(M_\infty,g_\infty)$,
the Riemannian manifold with cylindrical ends associated to
$(M,g_M)$; similarly, 
attach an infinite cylinder to the boundary of $N(\beta M)$
and extend the metric. 
Since the metrics near the boundary are of product type and since the metric on
the boundary is of psc, we are in the situation where the spin Dirac
operator on the manifolds with cylindrical ends  is invertible
at infinity. By Gromov-Lawson \cite[p.\ 117]{MR720933}
(the $b$-calculus of Melrose 
\cite{Melrose}  can alternatively be used here),
we then have a cylindrical $\alpha$-class 
$\alpha_{{\rm  cyl}} (M,g_M)$
in $KO_n$ \footnote{We are making 
a small abuse of notation in that we write $\alpha_{{\rm  cyl}} (M,g_M)$ and not 
$\alpha_{{\rm cyl}} (M_\infty,g_\infty)$.}. This class only depends 
on $M$ and on the metric $g_{\partial M}$ induced by $g_M$ on the
boundary. See Remark \ref{remark-stolz' R groups} for more
on this cylindrical class. Combining Gromov-Lawson and the analysis of
Albin-Gell-Redman, we also have a mixed class
$\alpha_{{\rm cyl},w} (N(\beta M),g_{N(\beta M)})$, also an
element in $KO_n$. This gives us the first part of the following Proposition:
\begin{proposition}\label{prop:gluing}
Under the above additional assumption, namely that $g|_{\p M}$ is of
psc, we have $KO$-classes
\[
\alpha_{{\rm  cyl}} (M,g_M), \quad \alpha_{{\rm cyl},w} (N(\beta
M),g_{N(\beta M)})\quad\text{in}\quad  KO_n
\]
and the following gluing formula holds:
\begin{equation}\label{gluing}
\alpha_{{\rm cyl}} (M,g_M) +  \alpha_{{\rm cyl},w} (N(\beta
M),g_{N(\beta M)})= \alpha_{w} (M_\Sigma,g) \quad\text{in}\quad KO_n. 
\end{equation}
\end{proposition}

%%The following is saved for Part II
%% \begin{proposition}
%% Under the same assumptions as above we also have index classes
%% associated to a Galois $\Gamma$-cover $\Gamma\to M^\Gamma_\Sigma \to M_\Sigma$:
%% $${\rm Ind}_{{\rm cyl}}  (\eth_\Gamma,M^\Gamma)\,,\quad {\rm Ind}_{{\rm cyl},w}  (\eth_\Gamma,N(\beta M)^\Gamma)\quad\text{in}\quad KO_* (C^*_{r,\mathbb{R}} \Gamma)$$
%% and the following gluing formula holds:
%% \begin{equation}\label{gluing-Ind}
%% {\rm Ind}_{w}(\eth_\Gamma,M^\Gamma_\Sigma)={\rm Ind}_{{\rm cyl}}  (\eth_\Gamma,M^\Gamma) + {\rm Ind}_{{\rm cyl},w}  (\eth_\Gamma,N(\beta M)^\Gamma)\quad\text{in}\quad KO_* (C^*_{r,\mathbb{R}} \Gamma).
%% \end{equation}
%% Consequently, if $(N(\beta M),_{N(\beta M)})$ if of psc then
%% \begin{equation}\label{w=cyl}
%% \alpha_{w} (M_\Sigma,g)= \alpha_{{\rm cyl}} (M,g_M) 
%% \quad\text{and}\quad {\rm Ind}_{w}(\eth_\Gamma,M^\Gamma_\Sigma)={\rm Ind}_{{\rm cyl}}  (\eth_\Gamma,M^\Gamma)\,.
%% \end{equation}
%% \end{proposition}
\begin{proof}
Only the gluing formula needs to be discussed. This follows from  a
small variation of a well known technique of Bunke, see \cite{Bunke}.
\end{proof}

\section{The case when the link is a homogeneous space}
\label{sec:setup}
\subsection{Geometric setup: motivation}
Here we specify a geometric setup relevant to the existence problem of
an adapted wedge psc-metric on a pseudomanifold $M_{\Sigma}$ with
fibered $L$-singularity.
Even in our case when the pseudomanifold
$M_{\Sigma}$ is spin, we have seen that an adapted wedge metric has to be
rather special even for an appropriate self-adjoint Dirac operator to
exist. This means that pure topological conditions on $M_{\Sigma}$ do
not give an appropriate setup for existence of a psc-metric.

We will bypass this issue by fixing an
appropriate geometrical structure on the pseudomanifolds we would like to
study. This will lead to the notion of
a \emph{pseudomanifold with fibered $(L,G)$-singularity}, where
$G$ is a Lie group acting transitively on a manifold $L$.

\subsection{Metric near the singular stratum}
Now we discuss some details of the geometry of the tubular
neighborhood $N=N(\beta M)$ of the singular stratum $\beta M$.  These
will be needed for explaining how we arrived at our definition of a
well-adapted metric, and will also be needed for the proof of Theorem
\ref{thm:psobstruction}.  Since the interior of $M$ will be irrelevant
here, we work simply with a bundle $\varphi_c\co
N\to B$ over an arbitrary base manifold $B=\beta M$, where the fibers
of $\varphi_c$ are cones on a fixed manifold $L$.
Since we want the geometry of $N$ to be well related to the geometry
of $L$ and of $B$, we assume that the bundle
$\varphi_c\co N\to B$ is the
associated bundle coming from a principal $G$-bundle $p\co P\to B$,
where $G$ is a compact Lie group acting transitively on $L$ by
isometries.  That means in particular that $L$ can be identified with
a homogeneous space $G/K$.  The case where $G$ is a torus has very
different behavior than the case where $G$ is semisimple, so we
restrict attention to the latter case in this paper.  (The case where
$G=S^1$ was studied in detail in \cite{BR}.)

Thus in this paper we take $G$ to be a compact connected semisimple
Lie group. In Section \ref{sec:main} we will
take $G$ to be simply connected; this amounts to a kind of ``spin''
condition on the $G$-bundle $p$ over $B$ (since it is
saying that the structure group of the link bundle
lifts to the universal cover), but we won't need this yet.
We fix a bi-invariant metric on $G$, or equivalently,
an $\text{Ad}$-invariant metric on the Lie algebra $\fg$
of $G$ (when $G$ is simple, this is necessarily a multiple of the Killing
form); in practice we will work with a constant multiple of this metric.
Then the tangent bundle of $L$ can be identified with $G\times_K \fp$,
where $\fp$ is the orthogonal complement of the Lie algebra $\fk$ of $K$
in $\fg$. (Note that $K$ acts on $\fp$ by the adjoint action.)
The space $\fp$ inherits an inner product from the inner
product on $\fg$, and thus defines a Riemannian metric on $L$
which will be fixed once and for all.  This metric is $G$-invariant
and has constant positive scalar curvature, and in fact nonnegative
(but not identically zero)
sectional curvature given by the formula \cite[\S5, p.\ 466]{MR0200865}:
\begin{equation*}
K(x\wedge y) = \frac14\Vert [x,y]_\fp\Vert^2 + \Vert [x,y]_\fk\Vert^2,
\end{equation*}
for $x,\,y$ orthonormal in $\fp$.  (This is the only place in this
paper where $K$ denotes sectional curvature, not the
isotropy group.  The notations $[x,y]_\fp$ and $[x,y]_\fk$ refer to
the orthogonal projections of the bracket of $x$ and $y$ into $\fp$
and $\fk$, respectively.)

Now suppose we have a principal $G$-bundle $p\co P\to B$, with base
$B$ a smooth compact manifold. We get an induced associated bundle
$\varphi\co \partial N\to B$, where $\partial N
= P\times_G (G/K) = P/K$.  The total space $\partial N$ of this
$L$-bundle will also be our manifold $\partial M$, for $M$ the
resolution of our (pseudo-)manifold with fibered
$(L,G)$-singularities.

We want to consider a Riemannian structure on $\partial N$ that is
adapted to this fibration structure. We construct this as follows.
Fix a connection $\nabla^p$ on the principal $G$-bundle $p\co P\to B$.
This induces a connection $\nabla^\varphi$ on the $L$-bundle
$\varphi\co \partial N\to B$. The tangent bundle of $\partial N$
splits as the \emph{vertical tangent bundle}, or \emph{tangent bundle
  along the fibers}, which is
\begin{equation*}
P\times_G TL = P\times_G (G\times_K \fp) = P\times_K \fp,
\end{equation*}
direct sum with the pull-back $\varphi^*(TB)$ of
$TB$.

Now we specify a metric on the bundle $\varphi_c\co N\to B$,
where we replace $L$ by the cone $c(L)$ on $L$, where
$c(L)=([0, R]\times L)/(\{0\}\times L)$ (the radius $R$ of
the cone (the distance to the vertex) will be determined later). 
%
%If we replace $L$ by the cone \textcolor{blue}{$c(L)$ on $L$,
%  i.e., $c(L)=([0, R]\times L)/(\{0\}\times L)$}, where the radius $R$
%of the cone (the distance to the vertex) will be determined later, we
%get the bundle $\varphi_{\textcolor{blue}{c}}\co N\to B$, where $N$ is
%a singular manifold.  (There will be no confusion in again denoting
%the bundle projection by $\varphi$, since the bundle $\partial N\to B$
%is obtained from this bundle simply by restricting attention to the
%boundary.)
We put a Riemannian metric $g_B$ on the base manifold $B$.  On the
complement of the singular stratum (diffeomorphic to $B$) in $N$, we
put the metric which is $(dr^2 + r^2g_L) \oplus \varphi^*_c g_B$,
where $g_L$ is the metric on $L$ defined above transported to the
vertical tangent bundle by making the vertical
fibers totally geodesic. Here the coordinate $r$ denotes the radial
distance from the singular stratum, $0<r\le R$, and $\varphi^*_c g_B$
is put on the horizontal fibers with respect to the connection
$\nabla^{\varphi_c}$. We note that the vertical metric had been
previously denoted by $g_{\partial M/\beta M}$.

Each vertical fiber of $\varphi_c$ is a totally geodesic metric cone on $L$,
with metric $dr^2 + r^2g_L$. Away from the cone point where $r=0$, this is
a warped product metric on $(0,R]\times L$, and so by
\cite[Lemma 3.1]{MR607943}, rederived (apparently independently)
in \cite[Proposition 7.3]{MR720933}, we have:
\begin{lemma}[{\cite[Lemma 3.1]{MR607943} and
      \cite[Proposition 7.3]{MR720933}}]
\label{lem:scalcurvcone}
The scalar curvature function $\kappa$ on each vertical fiber
$c(L)$ of $\varphi_c$ is $(\kappa_L - \kappa_\ell)r^{-2}$,
where $\kappa_\ell = \ell(\ell-1)$,
$\ell=\dim L$, is the scalar curvature of a standard round sphere
$S^{\ell}(1)$ of radius $1$.
\end{lemma}
Note that this is consistent with the fact that
the cone on a standard round sphere $S^{n-1}(1)$, with metric $dr^2 +
r^2g_{S^{n-1}(1)}$, is just flat Euclidean $n$-space.

Now we normalize the bi-invariant Riemannian
metric on $G$ (i.e., $\text{Ad}(G)$-invariant inner product on $\fg$)
so that the scalar curvature $\kappa_L$ comes
out to be exactly the constant $\kappa_\ell$. Because of Lemma
\ref{lem:scalcurvcone}, we obtain:
\begin{corollary}
\label{cor:scalcurvcone}  
With the normalization $\kappa_L=\kappa_\ell$, the conical vertical
fibers of the fibration $\varphi_c$ have scalar
curvature identically $0$, i.e., are scalar-flat.
\end{corollary}
This normalization is made to cancel out the contribution of the
scalar curvature of the fibers to the scalar curvature of an adapted
metric. If instead we had taken $\kappa_L > \kappa_\ell$, then we could
always get an adapted metric of {\psc} on $N$, and half of our problem
would go away.

Note that we still have at our disposal one more normalization, namely
the radius $R$ of the cones.  Since we are taking $\partial N$ to
correspond to $r=R$, that means the scalar curvature of each vertical
fiber $L$ of the bundle $\varphi\co \p N \to B$
%\footnote{ \textcolor{blue}{where, we recall, $\varphi
%    =\varphi_{\textcolor{blue}{c}}|_{\partial N}$}}
is $R^{-2}\kappa_\ell$.

At this point let's summarize the kind of Riemannian metrics we want
to deal with.  First, we fix a homogeneous space
$L=G/K$ with bi-invariant metric $g_L$ with constant scalar
curvature $\kappa_L=\ell(\ell-1)$. Then we we consider only
pseudomanifolds $M_\Sigma$ with $L$-fibered singularities whose
resolution $(M,\varphi\co \p M\to \beta M)$ has bundle $\varphi$ of
the very special form just discussed (i.e., it comes from a principal
$G$-bundle $p\co P\to \beta M$).  We will say that $M_\Sigma$ has
\emph{$(L,G)$-fibered singularities}.

\begin{definition}
\label{def:conicalmetric}
Assume $M_\Sigma$ has $(L,G)$-fibered singularities.  A
\emph{well-adapted} Riemannian wedge metric
on $M_\Sigma = M\cup_{\p M}([0,\varepsilon]\times \p M)\cup_{\p M}N$
is then given by
a Riemannian metric $g_M$ on $M$, a transition metric on the small
collar $[0,\varepsilon]\times \p M$, and a Riemannian metric $g_{\beta
  M}$ on the base of the fibration $\varphi\co \partial M\to \beta
M$. In addition, we assume the principal $G$-bundle $p\co P\to \beta
M$ has been equipped with a connection $\nabla^p$, which in turn
induces a connection $\nabla^\varphi$ on $\varphi$. On the tubular
neighborhood $N$ of $\beta M$, we put the metric $(dr^2 + r^2g_L)
\oplus \varphi_c^*\,g_B$ (\emph{orthogonal} direct sum), which 
is singular along the singular stratum $\beta M$. 
%\footnote{\textcolor{m}{It is actually kind of singular even when $L$
%    is a sphere;
%    in fact I would erase the last sentence.}}
%Here $\varphi_c^*\,g_B$ is put on the horizontal spaces for the
%connection $\nabla^{\varphi_c}$.

Furthermore, we require the Riemannian metrics to match to second
order along $\p M$, and we require $g_M$ to be a product metric in a
small collar neighborhood of the boundary $\p M$.  In the transition
region $[0,\varepsilon]\times \p M$, we take the metric to have the
form $(dr^2 + f(r)^2g_L) \oplus \varphi^*g_B$, where the $C^2$
function $f$ is given by $R+r$ for $0\le r \le \frac\varepsilon4$ and
by the constant $R + \frac\varepsilon2$ for $R+\frac{3\varepsilon}{4}
\le r\le R+\varepsilon$. In this way we get a $C^2$ interpolation
\footnote{The interpolation could be done in a different
way without making any essential difference, but it's convenient
to make a choice once and for all.}
between the metric on $N$ and the metric on $M$, without affecting
positivity of the scalar curvature.
\end{definition}
When the Riemannian metrics on $N$ and $\partial N$ are of the special
form given in Definition \ref{def:conicalmetric}, then
the bundles $\varphi_c\co (N\smallsetminus \beta M) \to
\beta M$ and $\varphi \co \p M \to \beta M$ are, in
fact, Riemannian submersions with fibers
$\mathring{c}(L)$ (the open cone on $L$, i.e., the
cone with the vertex removed) and $L$, respectively.  

These
submersions have totally geodesic fibers, and translation along the
fibers preserves the horizontal spaces since the splitting of
horizontal and vertical spaces comes from a connection on the
principal bundle $P$, so the O'Neill $T$-tensor (see
\cite[p.\ 460]{MR0200865}) vanishes for both of them.  We will need
the following:
\begin{proposition}[O'Neill]
\label{prop:OMeill}
Given a Riemannian submersion $\varphi\co X\to B$ of Riemannian
manifolds, with totally geodesic fibers all isometric to $F$, where a
Lie group acts transitively on $F$ and preserves horizontal spaces,
the scalar curvatures $\kappa_X$, $\kappa_F$, and $\kappa_B$ are
related as follows:
\[
\kappa_X = \kappa_F + \kappa_B + \sum_{j, m}\Vert A_{x_j}v_m\Vert^2
- 3 \sum_{j\ne k}\Vert A_{x_j}x_k\Vert^2.
\]
Here $\{x_j\}_{j\le\dim B}$
is an orthonormal frame for the horizontal tangent space
and $\{v_m\}_{m\le\dim F}$ is an orthonormal frame for the
vertical tangent space.
\end{proposition}
\begin{proof}
  Since the $T$-tensor vanishes identically under the conditions
  of the proposition, this follows immediately from
  \cite[Corollary 1, p.\ 465]{MR0200865} after summing over
  all pairs of distinct elements of the orthonormal frame
  $\{x_j,\,v_m\}_{j\le\dim B,\,m\le\dim F}$.
\end{proof}
As a consequence of Proposition \ref{prop:OMeill}, we obtain the following
theorem, which is an important part of Theorems \ref{thm:psobstruction}
and \ref{thm:pscsufficiency}.
\begin{theorem}
\label{thm:bundlepsc}
Let $L=G/K$ be a homogeneous space of a compact connected semisimple
Lie group $G$, equipped with a $G$-invariant metric of scalar
curvature $\kappa_L=\kappa_\ell$ as above, and let $M_\Sigma$ be a
compact manifold with $(L,G)$-fibered singularities,
with resolution $(M,\varphi\co \p M\to \beta M)$, where the $L$-bundle
comes from a principal $G$-bundle $p\co P\to \beta M$. Then the
following hold.
\begin{enumerate}
\item $\p M$ always has a well-adapted Riemannian metric of \psc.
\item If $\beta M$ has a metric of \psc, then its tubular neighborhood
  $N$ has a well-adapted metric of \psc.
\item If the tubular neighborhood $N$ of $\beta M$ has a well-adapted
  metric of \psc, then $\beta M$ has a metric of \psc.
\end{enumerate}
\end{theorem}
\begin{proof}
  We start with an observation about well-adapted metrics, which is that
  for $x$ a horizontal vector and $v$ a vertical vector, $A_x(v)=0$.
  The reason is that if we choose a geodesic $\gamma$ with
  $\dot\gamma(0)=x$, then by the construction of adapted metrics,
  $\dot\gamma(t)$ stays horizontal and the parallel transport of $v$ along
  $\gamma$ remains vertical.  Thus the horizontal component of
  $\nabla_{\dot \gamma}(v)$ vanishes, and $A_x(v)=0$.  This removes
  one error term from the formula of Proposition \ref{prop:OMeill}.
  
  (1). If $L=G$, so that $\p M$ has a free action of $G$, then this is
  an easy special case of the main result of \cite{MR0358841}. In general,
  this part is codified as the Observation in
  \cite[p.\ 512]{MR1189863}. We can deduce it from Proposition
  \ref{prop:OMeill} by choosing $R$ very small (recall that in effect
  we are taking the fibers of $\varphi$ to be copies of $L$ with
  the diameter multiplied by a factor of $R$), and thus with scalar
  curvature $R^{-2}\kappa_\ell$.  This then swamps all the other terms
  on the right in Proposition \ref{prop:OMeill}.

  (2). Let's look at the formula of Proposition \ref{prop:OMeill} for
  the scalar curvature of $N$. By our assumption on the metric on $L$,
  the scalar curvature of the $\mathring{c}(L)$
  fibers vanishes.  Let's rescale the metric on $\beta M$ by
  multiplying lengths of vectors by $t$.  Note that $\Vert
  A_{x_j}x_k\Vert$ is computed in the metric of the vertical fibers,
  which we are keeping fixed.  However, when we rescale the metric
  from $g_{\beta M}$ to $t^2g_{\beta M}$, our orthonormal frame
  changes from $\{x_j,\,v_m\}_{j,m}$ to $\{t^{-1}x_j,\,v_m\}_{j,m}$.
  Thus with the rescaled metric, the scalar curvature of $N$ becomes
  \[
  t^{-2}\kappa_B   - 3 \sum_{j\ne k}t^{-4}\Vert A_{x_j}x_k\Vert^2.
  \]
  Now note that if the error term on the right were zero,
  we would have $\kappa_B>0$ and if and only if $\kappa_N>0$.
  This is not quite true for the original metric if the error
  term is nonzero, but if we let $t\to\infty$, the negative
  term (involving $\Vert A_{x_j}x_k\Vert^2$) goes to zero fastest.
  So if $\kappa_B>0$, eventually the rescaled value of $\kappa_N$ becomes
  positive. This proves (2).

  (3). Suppose $\kappa_N$ is strictly bigger than $0$. In our situation,
  we have
  \[
  \kappa_N = \kappa_B -(\text{something nonnegative}),
  \]
  so $\kappa_B>0$.
\end{proof}  

\begin{example}
\label{ex:K3bundle}
In this example, we take $L=G=\SU(2)=S^3$.
Let $\beta M$ be a K3 surface, a spin simply connected $4$-manifold
with Betti numbers $1,0,22,0,1$ and $\widehat  A(K3)=2$.
In particular, $\beta M$ does not admit a metric of {\psc}.
However, there is a one-parameter family of principal $S^3$-bundles $P$
over $\beta M$, classified by the second Chern class $c_2$.  (Indeed, since 
$[S^4, \bH\bP^\infty]\cong \pi_3(S^3)\cong \bZ$, we have such a
family of bundles over $S^4$, and we just pull them back via a map
$\beta M\to S^4$ of degree $1$.)  The total space $\p M$ of each such
bundle is a simply connected spin $7$-manifold, and since
$\Omega^\spin_7=0$, it is a spin boundary, say the boundary of an
$8$-manifold $M$.

The manifold $M_\Sigma$ obtained by gluing the disk bundle $D\cL$ of
the quaternionic line bundle $\cL$ associated to
the principal bundle $P$ to $M$, for any choice of $c_2$, is a simply
connected spin $8$-manifold.  We can take it to be the double of
$D\cL$ along $\beta M$, which will have $\widehat A$-genus $0$, and
then it admits a metric of {\psc}.  However, it does not admit a
\emph{well-adapted} metric of {\psc}, since Theorem
\ref{thm:bundlepsc}(3) would imply then that $\beta M$ admits a
metric of {\psc}, which it does not.
\end{example}
\begin{example}
\label{ex:eta2}
We again take $L=G=S^3$.  Here is another nontrivial example of a spin
manifold with $S^3$-fibered singularities.  Note that the set of
quaternionic line bundles $\cL$ on $S^n$ is parametrized by $[S^n,
  B(\Sp(1))] = \pi_{n-1}(S^3)$, which is finite for all $n\ge 5$ and
usually nonzero.  Take for example $n=6$.  Then $[S^6, B(\Sp(1))] =
\pi_5(S^3)\cong \bZ/2$, with the generator represented by $\eta^2$
($\eta$ the generator of the stable $1$-stem, as usual).  So there are
two distinct quaternionic line bundles $\cL$ over $S^6$, both of which
are rationally trivial (since $c_2$ must vanish).  For each of them,
$\partial(D\cL)$ is a principal $S^3=\Sp(1)$-bundle over $S^6$,
rationally equivalent to $S^3\times S^6$. (When $\cL$ is trivial,
$\partial(D\cL)=S^3\times S^6$, and when $L$ is nontrivial,
$\partial(D\cL)= (S^3\vee S^6)\cup e^9$, with the attaching map of the
$9$-cell given by $S^8\xrightarrow{\eta^2} S^6\hookrightarrow S^3\vee
S^6$.)  In both cases, this manifold $\partial M$, since it carries a
free $S^3$-action, admits positive scalar curvature \cite{MR0358841}
and thus has trivial $\alpha$-invariant in $ko_9\cong
  \bZ/2$.  It is a spin boundary, since we can fill in each $S^3$
fiber over $S^6$ with a $4$-disk, and we can get a simply connected
spin manifold $M^{10}$ with boundary $\partial M$.  After gluing in
$D\cL$, we have a closed simply connected singular spin $10$-manifold
$M_\Sigma$, which we can take to be the double of $D\cL$ along its
boundary.  This admits positive scalar curvature, even if we require
our metric to be well adapted in the sense of Definition
\ref{def:conicalmetric}, by Theorem \ref{thm:bundlepsc}.

Part of the reason why this example is interesting is that if the
bundle $p$ is nontrivial, then the classifying map
$S^6\to \bH\bP^\infty$ represents a nontrivial class in
$\widetilde{ko}_6(\bH\bP^\infty)\cong ko_2 \cong \bZ/2$
(see the following Theorem \ref{thm:koHP}), since the
generator of $ko_2$ is represented by $\eta^2$.
\end{example}
We will need as a technical tool the following theorem pointed
out to us by Bob Bruner:
\begin{theorem}[Bruner--Greenlees]
\label{thm:koHP}  
  $ko_*(\bH\bP^\infty)$ is a free
$ko_*$-module on generators $z_j$ in dimensions
$4j$, $j\in\bN$.  The action of $ko^*(\bH\bP^\infty)\cong
ko^*[[z]]$, $\dim z = 4$, is by $z\cdot z_j = z_{j-1}$.
\end{theorem}
\begin{proof}
  There are a few possible proofs.  One is to show by induction on $n$ that
  the Atiyah-Hirzebruch spectral sequence (AHSS)
  \[
  H_p(\bH\bP^n, ko_q) \Longrightarrow ko_{p+q}(\bH\bP^n)
  \]
  collapses at $E_2$, which is proved in
  \cite[Lemmas 2.4 and 2.5]{MR0254841},
  or alternatively, that the boundary map in
  the exact cofiber sequence
  \[
  \cdots \to ko_j(\bH\bP^{n-1}) \to ko_j(\bH\bP^n) \to
  \widetilde{ko}_j(S^{4n}) \to \cdots
  \]
  vanishes, which is proved in \cite[Corollary 3.1]{Thakur}.

  The proof suggested to us by Bruner is a bit slicker; it is
  shown in \cite[p.\ 86, Theorem 5.3.1]{MR2723113} that
  $ko^*(\bH\bP^\infty) = ko^*[[z]]$ for $z$ in dimension $4$. The
  additive result follows by the universal coefficient theorem and the
  $ko^*[[z]]$-module structure follows from the local cohomology
  spectral sequence \cite{MR1888194}.
\end{proof}
\section{Relevant Bordism Theory}
\label{sec:main}

\subsection{Bordism of pseudomanifolds}
It is known that a meaningful bordism theory in the framework of
stratified pseudomanifolds requires some restrictions on their
stratification and the equivalence relation, see, say,
\cite{Banagl-book}. On the other hand, under some natural restrictions,
such bordism groups could be highly interesting. For instance, the
Witt-bordism groups are such, and they emerged naturally in the
contexts related to the signature and to the signature operator
\cite{MR714770}.
In our geometrical context, we will consider the following two
bordism groups: $\Omega^{\spin, L\fb}_*$ and $\Omega^{\spin, (L,G)\fb}_*$.

We start with $\Omega^{\spin, L\fb}_*$. Let $\Sigma=(L)$, and let
$M_{\Sigma}$ and $M_{\Sigma}'$ be two pseudomanifolds with fibered
$L$-singularities. Then a bordism $W_{\Sigma} \co
M_{\Sigma}\rightsquigarrow M_{\Sigma}'$ is a pseudomanifold
$W_{\Sigma}$ with \emph{boundary} and with fibered $L$-singularities.
This means that $W_{\Sigma}=W\cup -N(\beta W)$, where the
resolving manifold $W$ is a spin manifold with corners and the boundary
$\p W$ of the resolution is given a splitting
$\p W =\p^{(0)} W \cup \p^{(1)} W$, where
$$
\p^{(0)}W = M\sqcup -M', \ \ 
\p (\p^{(1)} W) = \p M \sqcup - \p M', \ \ \
\p(\beta W) = \beta M\sqcup -\beta M', \ \ 
$$
i.e., $\p^{(1)} W \co \p M\rightsquigarrow \p M'$ and 
$\beta W \co \beta M\rightsquigarrow \beta M'$ are usual
bordisms between closed spin manifolds, and
$\p^{(0)} W \cap \p^{(1)} W = \p M \sqcup - \p M'$.
Furthermore, it is assumed that the
fiber bundle $F \co  \p^{(1)}W \to \beta W$
restricts to the fiber bundles $f\co \p M \to \beta M$ and
$f'\co \p M' \to \beta M'$ respectively.
This gives a well-defined bordism group $\Omega^{\spin, L\fb}_*$ and,
in fact, a bordism theory $\Omega^{\spin, L\fb}_*(-)$. 

\begin{remark}
It is worth noticing that, in the above setting, the bordism group
$\Omega^{\spin, L\fb}_*$ is rather complicated. Indeed, according to
our definition, an $L$-fibration $f\co \p M \to \beta M$ is just a smooth
fiber bundle, and thus such a fibration is classified by a map
$\beta M \to \BDiff(L)$. Then the  correspondence
$M_{\Sigma} \mapsto (\beta M \to \BDiff(L))$
defines a Bockstein homomorphism
$\beta\co \Omega^{\spin, L\fb}_*\to \Omega^{\spin}_{*-\ell-1}(\BDiff(L))$.
There is also a transfer homomorphism
$\tau\co \Omega^{\spin}_{*}(\BDiff(L))\to \Omega^{\spin}_{*+\ell}$ which
takes a bordism class $B\to \BDiff(L)$ to the corresponding smooth
$L$-fiber bundle $E\to B$. All together, these fit into an exact
triangle of homology
theories
\begin{equation}\label{eq:bordism-general}
  \begin{diagram} \setlength{\dgARROWLENGTH}{1.95em}
    \node{\Omega^{\spin}_*(-)} \arrow[2]{e,t}{i}
    \node[2]{\Omega^{\spin,L\fb}_*(-)} \arrow{sw,b}{\beta} \\
    \node[2]{\Omega^{\spin}_*(\BDiff(L)\wedge-)} \arrow{nw,b}{\tau}
  \end{diagram}
\end{equation}
where $i$ is a transformation considering each spin manifold as a
pseudomanifold with empty singularities. Thus the complexity of the
bordism group $\Omega^{\spin,L\fb}_*$ is determined by the classifying
space $\BDiff(L)$, which is known to be very complicated for almost
every smooth manifold $L$.
\end{remark}

\subsection{Spin manifolds with fibered $(L,G)$-singularities}
We recall our definition of a closed manifold with fibered
$(L,G)$-singularities. For this we need to fix the following
topological data:
\begin{enumerate}
\item[(i)] a closed spin smooth manifold $L$ (a link, as above),
  $\dim L=\ell$;
\item[(ii)] a Lie group $G$ mapping to $\Diff(L)$ with a lifting
  of $G$ to an action on the principal spin bundle. (This is
  necessary for some things we will do later involving the Dirac operator.)
\end{enumerate}  
These data are good enough to construct a relevant bordism theory;
however, we would also like to fix some geometrical data, namely a
Riemannian metric $g_L$ on the link $L$ such that
\begin{enumerate}
\item[(iii)] the scalar curvature $\kappa_{L}$ of the metric $g_L$
  is the positive constant $\kappa_\ell$ and the Lie group $G$ acts
  transitively by a subgroup of the
  isometry group $\mathrm{Isom}(L,g_L)\subset \Diff(L)$.
\end{enumerate} 
We denote by $BG$ the corresponding classifying space and by $EG\to
BG$ the universal principal bundle. We say that
$\pi\co E\to B$ is an $(L,G)$-\emph{fiber bundle} if
it is a smooth fiber bundle with fiber $L$ and structure group
$G$. There is a universal $(L,G)$-fiber bundle $E(L)\to BG$, where
$E(L):=EG\times_G L$.  Then for any $(L,G)$-fiber bundle $\pi\co E\to B$,
there is a classifying map $f\co B\to BG$ such that $E=f^*E(L)$.

As above, we denote by $c(L)$ the cone over $L$ and
let $G$ act on $c(L)$ slice-wise.  We obtain an associated
$(c(L),G)$-fiber bundle $\pi_c \co N(B)\to B$,
where $N(B) = f^* (EG\times_G c(L))$.  Note that if $B$ is a closed
manifold, then $N(B)$ is a singular manifold with boundary $\p N(B)=E$
and singular set $B$. We obtain a commutative diagram of fiber
bundles:
\begin{equation}\label{eq:1}
  \begin{diagram}
    \setlength{\dgARROWLENGTH}{1.95em}
  \node{E = \p N(B)}
           \arrow{e,t}{i}
           \arrow{s,l}{\pi}
  \node{N(B)}
             \arrow{s,l}{\pi_c}
  \\
  \node{B}
  \arrow{e,t}{Id}
  \node{B}
  \end{diagram}
\end{equation}
where the fiber $L$ is identified with the boundary of the cone $c(L)$.

Let $\displaystyle M_{\Sigma}=M\cup_{\p M}N(\beta M)$,
where $M$ is a spin manifold with boundary $\p M$, where $\p M$
is the total space of an $(L,G)$-fiber bundle $\varphi\co \p M\to \beta  M$
given by its structure map $f\co \beta M\to BG$. 
$\displaystyle M_{\Sigma}$ is a \emph{closed} $(L,G)$-singular spin
pseudomanifold.
\begin{definition}\label{def:4-2}
We say that a closed $(L,G)$-singular spin pseudomanifold
$M_{\Sigma}=M\cup_{\p M}N(\beta M)$ as above is an $(L,G)$-\emph{boundary} if
there is a spin manifold $\bar M$ (with corners) such that
\begin{enumerate}
\item[(1)] $\p \bar M = \p^{(0)} \bar M  \cup \p^{(1)} \bar M $, where 
$\p^{(1)} \bar M=M$ (i.e., we have chosen a diffeomorphism $\p^{(1)} \bar
M\cong M$);
\item[(2)] $\p^{(0)} \bar M $ is a total space of an $(L,G)$-fiber bundle 
$\bar \varphi\co \p^{(0)} \bar M \to \beta \bar M $
given by a structure map $\bar f\co \beta \bar M \to BG$ with the
restriction $\p(\beta \bar M )=\beta M$ and $\bar
f|_{\p(\beta \bar M )}= f$;
\item[(3)] $\p(\p^{(0)} \bar M )=\p M$ and the
restriction $\bar \varphi|_{\p(\p^{(0)} \bar M )}= \phi$.
\end{enumerate} 
Then the space
$\displaystyle\bar M_{\Sigma}=\bar M \cup_{\p \bar M }N(\beta \bar M)$
is a \emph{spin manifold
with fibered $(L,G)$-singularities with boundary
$\delta \bar M_{\Sigma}=M_{\Sigma}$.} For short, we say that $\bar M_{\Sigma}$
is an \emph{$(L,G)$-singular spin pseudomanifold with boundary
$\delta \bar M_{\Sigma}=M_{\Sigma}$}.
\end{definition}
Now we say that two $(L,G)$-singular spin pseudomanifolds $M_{\Sigma}$,
$M_{\Sigma}'$ are bordant if there exists an $(L,G)$-singular spin
pseudomanifold $W_{\Sigma}$ with boundary
$\delta W_{\Sigma}=M_{\Sigma}\sqcup -M_{\Sigma}'$. Given a bordism as above,
we use the notation
$W_{\Sigma}\co M_{\Sigma}\rightsquigarrow M_{\Sigma}'$. This
determines corresponding bordism groups 
$\Omega_n^{\spin, (L,G)\fb}$ and a bordism theory
$\Omega_*^{\spin, (L,G)\fb}(-)$. We now give
a few more details in this direction.

\begin{definition}\label{def:4-3}
Let $M_{\Sigma}=M\cup_{\p M}N(\beta M)$ and let
$f\co \beta M \to BG$ be the corresponding structure map, so that
we have a commutative diagram
\begin{equation}\label{eq:3}
  \begin{diagram}
    \setlength{\dgARROWLENGTH}{1.1em}
  \node{\p M}
           \arrow{e,t}{\hat f}
           \arrow{s,l}{\phi}
  \node{E(L)}
             \arrow{s,l}{\phi_0}
  \\
  \node{\beta M}
  \arrow{e,t}{f}
  \node{BG}
  \end{diagram}
\end{equation}
A map $\xi\co M\to X$ \emph{determines an element in}
$\Omega_*^{\spin, (L,G)\fb}(X)$ if the restriction
$\xi|_{\p M}$ coincides with the composition
$pr\circ (\hat f \times \xi_{\beta})$,
where $\xi_{\beta}\co \beta M \to X$ is some map and the map
$(\hat f \times \xi_{\beta})$ is given by the following commutative diagram:
\begin{equation}\label{eq:4}
  \begin{diagram}
    \setlength{\dgARROWLENGTH}{2.05em}
  \node{\p M}
           \arrow{e,t}{\hat f \times \xi_{\beta}}
           \arrow{s,l}{\phi}
  \node{E(L)\times X }
             \arrow{s,r}{\phi_0\times Id}
              \arrow{e,t}{pr}
  \node{X}            
  \\
  \node{\beta M}
  \arrow{e,t}{f\times \xi_{\beta}}
  \node{BG\times X}
  \end{diagram}
\end{equation}
Here $pr$ is a projection on the second factor. We note that, by
definition, there is a canonical extension of such a map $\xi\co M\to X$
to a map $\xi_{\Sigma} \co M_{\Sigma}\to X$.
\end{definition}
Let $i\co\Omega^{\spin}_*(-)\to \Omega^{\spin, (L,G)\fb}_*(-)$ be the natural
transformation given by considering a
closed manifold as a pseudomanifold with empty $(L,G)$-singularity. Then the
natural transformation $\beta\co \Omega^{\spin,(L,G)\fb}_*(-)\to
\Omega^{\spin}_{*-\ell-1}(BG\wedge -)$ is given by the correspondence
$M_{\Sigma} \mapsto (f\co\beta M \to BG)$. Finally, the
natural transformation $T\co\Omega^{\spin}_{*}(BG\wedge -)\to
\Omega^{\spin}_{*+\ell}(-)$ is a transfer map which takes a map
$f\co  B\to BG$ to the induced $(L,G)$-bundle $E\to B$.
\begin{proposition}
\label{prop:bordismtriangle}
There is an exact triangle of bordism theories:
\begin{equation}\label{eq:5}
  \begin{diagram}
    \setlength{\dgARROWLENGTH}{1.95em}
  \node{\Omega^{\spin}_*(-)}
          \arrow[2]{e,t}{i}
  \node[2]{\Omega^{\spin,(L,G)\fb}_*(-)}
          \arrow{sw,t}{\beta}
  \\
  \node[2]{\Omega^{\spin}_{*}(BG\wedge -)}
  \arrow{nw,t}{T} \node {.}
  \end{diagram}
\end{equation}
\end{proposition}
\begin{proof}
As usual, that the composites $i\circ T$, $\beta\circ i$, and
$T\circ \beta$ vanish is clear.  To prove exactness, the coefficients
$(-)$ ``come along for the ride,'' so we just give the proof without
them for simplicity of notation. If $T([B\to BG])=0$, that means that
the total space $E$ of the induced $(L,G)$-bundle $E\to B$ is a
spin boundary, so we can write $E=\p M$ for some spin manifold $M$,
which means that $[B=\beta M\to BG]$ is in the image of $\beta$.
If $\beta([M_\Sigma])=0$, that means we have $(M, \p M)$ with
an $(L,G)$-bundle $\p M\to \beta M$ so that the underlying principal
$G$-bundle $G\to P\to \beta M$ bounds in $\Omega^{\spin}_{*}(BG)$.
Suppose we have a spin manifold $N$ with boundary, mapping to $BG$,
so that $\p N = \beta M$ and the principal $G$-bundle $P' \to N$
extends $P\to \beta M$.  Then $P' \times_G L = M'$ bounds $\p M$, and
$M_\Sigma$ is bordant as a spin manifold
with $(L,G)$-singularities to the closed manifold $M\cup_{\p M} -M'$,
with $M'\to N$ providing the bordism, and hence $[M_\Sigma]$
hence comes from the image of $i$.  Finally, suppose we have
a closed spin manifold $M$ with $i([M])=0$. That means there is a spin
$(L,G)$-pseudomanifold $W_\Sigma$ with $\delta W_\Sigma = M$.
Resolving, we have a spin manifold $W$ (with boundary\footnote{Usually
there would be corners, but here the corner
$\partial^{(0)}W \cap \partial^{(1)}W$ is empty.}) with
$\partial W = \partial^{(0)}W \sqcup \partial^{(1)}W$, where
$\partial^{(0)}W$ is closed (since $\p M=\emptyset$),
$\partial^{(1)}W = M$, and $\partial^{(0)}W$ is is the total space of
an $(L,G)$-bundle $\partial^{(0)}W\to \beta W$, with $\beta W$ also
closed since $\beta M=\emptyset$.  This shows $M$ is spin bordant
to $-\partial^{(0)}W$, and thus $[M] = T(-[\beta W\to BG])$.
\end{proof}
\subsection{Bordism theorem}
One of our goals is to push a well-adapted metric of {\psc} through an
$(L,G)$-singular bordism. Here is the key result:
\begin{theorem}[Bordism Theorem]
\label{thm:4-1}
Assume $G$ is a simply connected Lie group and $L$ is spin.
Let $M_{\Sigma}$, $M'_{\Sigma}$ be two $(L,G)$-singular spin
pseudomanifolds of dimension $n\geq 6 + \ell$
representing the same class $x\in \Omega^{\spin,(L,G)\fb}_{n}$, with
$M$ and $\beta M$ simply connected.
Assume $M'_{\Sigma}$ has a well-adapted psc-metric
$g'$. Then there exists an $(L,G)$-bordism $W_{\Sigma}\co
M_{\Sigma}\rightsquigarrow M_{\Sigma}'$ together with a well-adapted
psc-metric $\bar g$ which is a product metric near the boundary
$\delta  W_{\Sigma}= M_{\Sigma}\sqcup - M_{\Sigma}'$ such that
$\bar g|_{M_{\Sigma}'}=g'$. In particular, $M_{\Sigma}$ admits a well-adapted
psc-metric $g$.
\end{theorem}
This result will follow from a purely topological result which
is just surgery-theoretic:
\begin{theorem}[Surgery Theorem]
\label{thm:surg}
Assume $G$ is a simply connected Lie group and $L$ is spin.
Let $M_{\Sigma}$, $M'_{\Sigma}$ be two $(L,G)$-singular spin pseudomanifolds
of dimension $n \geq 6 + \ell$
representing the same class $x\in \Omega^{\spin,(L,G)\fb}_{n}$.
Assume that $M_{\Sigma}=M\cup_{\p M} N(\beta M)$,
$M_{\Sigma}'=M'\cup_{\p M'} N(\beta M')$ with corresponding structure
maps $f \co\beta M \to BG$ and $f' \co\beta M \to BG$.
Also assume that $M$ and $\beta M$ are spin and simply
connected.\footnote{This implies that $M_\Sigma$ is $1$-connected, since 
$\beta M$ and $L$ $1$-connected imply that $\p M$ is $1$-connected, 
and by Van Kampen's Theorem, 
$\pi_1(M_\Sigma) = \pi_1(M) *_{\pi_1(\p M)} \pi_1(\beta M)$.}
Then there exists an $(L,G)$-bordism
$W_{\Sigma}\co M_{\Sigma}\rightsquigarrow M_{\Sigma}'$,
$W_{\Sigma}= W\cup_{\p W} N(\beta W)$, with a structure map
$\bar f \co \beta W\to BG$, such that
$(\beta W, \beta M)$ and $(W, M)$ are $2$-connected.
\end{theorem}
\begin{proof}[Proof of Theorem \ref{thm:surg}]
Start with any $(L,G)$-bordism
$W_{\Sigma}\co M_{\Sigma}\rightsquigarrow M_{\Sigma}'$.  Recall that this comes
with a map $F\co \beta W\to BG$ restricting to the given bundle
data $f\co \beta M\to BG$ and $f'\co \beta M'\to BG$.
First we modify
$W$ by surgery to reduce to the case where $(\beta W, \beta M)$ is
$2$-connected.  Recall that we are assuming $\beta M$ and $M$ are simply
connected, though we make no such assumption on $\beta M'$ and $M'$.
We begin by killing $\pi_1(\beta W)$ through surgery. Recall that we are
assuming that $\dim(\beta M)\ge 5$, so $\dim(\beta W)\ge 6$.
Given any class in $\pi_1(\beta W)$, we can represent it by an
embedded circle, which will have trivial normal bundle.  Since
$G$ is simply connected, $BG$ will actually be $3$-connected (any Lie
group has vanishing $\pi_2$, and we are assuming $\pi_1(G)=0$).
So we can do surgery on this circle so that $F$ extends over the trace
of the surgery, which will be a manifold $V$ with boundary
which we can attach to $\beta W$.  Thus we can suppose that $\beta W$
is simply connected. Next, look at the exact sequence
$\pi_2(\beta W) \to \pi_2(\beta W, \beta M)\to \pi_1(\beta M)=0$.  If
$\pi_2(\beta W, \beta M)\ne 0$, we can represent any generator of this
group by an embedded $S^2$ in $\beta W$.  Since everything is spin,
this $2$-sphere has trivial normal bundle.  Again, the map
$F$ from this $2$-sphere to $BG$ is null-homotopic since $BG$
is $2$-connected (even $3$-connected). So again we can do surgery
so that $F$ extends over the trace of the surgery.  After attaching
the traces of all surgeries needed to $\beta W$, we have reduced
to the case where $(\beta W,\beta M)$ is $2$-connected.

Te next step is to do something similar on the interior of $W$
to make $(W,M)$ $2$-connected.  The argument is exactly the same.
Note by the way that whatever changes we make in $W$ can automatically
be lifted up to changes in $W_{\Sigma}$ by lifting surgeries on
$\beta W$ to modifications of $N(\beta W)$.
\end{proof}
\begin{proof}[Proof of Theorem \ref{thm:4-1}
from Theorem \ref{thm:surg}]
Apply Theorem \ref{thm:surg} and assume we have a bordism $W$
with $(W,M)$, $(\beta W, \beta M)$ $2$-connected.  That means
that we can decompose the bordism into a sequence of surgeries
on $(M', \p M'\to \beta M')$ (compatible with the map to $BG$
on the Bockstein) which are always in codimension $3$ or more.
Then we apply the Gromov-Lawson surgery theorem
\cite{MR577131}, first on the Bocksteins, to push a psc metric
on $\beta M'$ to one on $\beta M$.  Since the bordism is compatible
with the maps to $BG$, we get a well-adapted metric of {\psc}
on the tubular neighborhood of $\beta W$.  The next step is to
push the psc metric on the interior of $M'$ to one on the interior
of $M$.  For this we use the Gromov-Lawson surgery theorem
again, possible since $(W,M)$ is $2$-connected. 
\end{proof}
\section{$KO$-obstructions on  $(L,G)$-fibered pseudomanifolds}
\label{sec:obstructions}
Let $M_\Sigma$ be a pseudomanifold with
$(L,G)$-fibered singularities. Let $f\co\beta M\to BG$ be the associated
classifying map and let $g$ be a well-adapted wedge metric on
$M_\Sigma$. Recall that by definition, the restriction $g_{\p M} =
g|_{\p M}$ is consistent with the natural vertical metric
$g_{\p   M/\beta M}$ on $\p M$ induced by the psc-metric $g_L$ on the link
$L$. See Definition \ref{def:conicalmetric} for details. 

Recall from Theorem \ref{thm:bundlepsc} that $g_{\p M}$ is a
psc-metric: indeed we can rescale the fiber metric $g_L$ to achieve
that.  Moreover, if $g_{\beta M}$ is a psc-metric on $\beta M$, then,
up to rescaling, $g_{N(\beta M)}$ is also a psc-metric on $N(\beta
M)$.  Vice versa, if $g_{N(\beta M)}$ is a psc-metric, then
$g_{\beta   M}$ is also a psc-metric. 

Denote, as above, $g_M=g|_{M}$. We know that in this setting the
$KO$-classes $\alpha_{{\rm cyl}}(M,g_M)$ and $\alpha_w (M_\Sigma,g)$
are well defined. 
Moreover, Proposition \ref{prop:gluing} implies that
these classes coincide for any pseudomanifold $M_\Sigma$ with
$(L,G)$-fibered singularities, provided  $g_{N(\beta M)}$ is a psc-metric.

Assume now that $g$ on $M_\Sigma$ is a psc-metric. Then, obviously,
the metric $g_{N(\beta M)}$ is also psc and so is $g_{\beta (M)}$.  This
implies that
\begin{equation}\label{beta-obstruction}
\alpha (\beta(M),g_{\beta M})=0\quad\text{in}\quad KO_{n-\ell-1}\,.
\end{equation} 
and 
\begin{equation}\label{obstruction}
\alpha_{{\rm cyl}} (M,g_M)= \alpha_w (M_\Sigma,g)=0\;\;\text{in}\;\; KO_n\, .
\end{equation}
with the first equality in \eqref{obstruction} following from Proposition \ref{prop:gluing}, as we have already remarked,
and the second one from the classic results of Gromov-Lawson
or from Theorem \ref{theo:geometric-witt-part1}, item (4).

Formulae \eqref{beta-obstruction} and \eqref{obstruction} prove the
obstruction theorem (Theorem \ref{thm:psobstruction} stated in the
Introduction).

Note, however, that the class
$f_* [\Di_{\beta M}]\in KO_{n-\ell-1} (BG)$ is \emph{not} an obstruction.
Example \ref{ex:eta2} is a counterexample.

These obstructions are in fact obtained from suitable group
homomorphisms, as we shall now explain.
\begin{proposition}\label{prop:bordism-map-alpha}
Let $\Omega^{\spin,(L,G)\fb}_*$ be the
bordism group as above. Then we have well defined homomorphisms:
\begin{equation}\label{bordism-hom}
\alpha_{{\rm cyl}} \co \Omega^{\spin,(L,G)\fb}_* \to KO_*  \quad\text{and}\quad 
\end{equation}
\begin{equation}\label{bordism-hom-2}
\alpha_{\beta M}\co \Omega^{\spin,(L,G)\fb}_* \to KO_{*-\ell-1} \,.
\end{equation}
\end{proposition}
\begin{proof}
First of all, we have to define $\alpha_{{\rm cyl}}  [M_\Sigma]$,
with $[M_\Sigma]\in \Omega^{\spin,(L,G)\fb}_*$. 
We recall that $M_\Sigma=M\cup_{\p M} N(\beta M)$ and 
set
\[
\alpha_{{\rm cyl}}  [M_\Sigma]:= \alpha_{{\rm cyl}} (M, g_M)
\]
with $g$ a well-adapted wedge metric on the regular part of $M_\Sigma$
and $g_M:= g|_{M}$, as usual. Here, because  of the very definition 
of well-adapted wedge metric on a manifold with $(L,G)$-fibered
singularities, the homomorphism is well defined, independent of the choice 
of $g$. Indeed, we know that if $g$ and $g'$ are wedge metrics,
then $g_{\p M}$ and $g'_{\p M}$
are of psc and with the same vertical metric,
the one induced by the natural metric on $L$.
Consider an arbitrary path of wedge metrics joining
$g$ and $g'$, call it $\{g(t)\}_{t\in [0,1]}$. 
Remark that the family $\{g(t)\}_{t\in [0,1]}$ restricts to a
family of submersion metrics
$\{g(t)|_{\p M}\}_{t\in [0,1]}$ on $\p M$ and the latter fixes a
Riemannian metric $g_{\p M\times [0,1]}$
on $\p M \times [0,1]$ that 
we can assume to be of product-type near the boundary. Then,
always from Bunke \cite{Bunke}, we have:
\[
\alpha_{{\rm cyl}} (M, g_M) -  \alpha_{{\rm cyl}} (M, g'_M)
=  \alpha_{{\rm cyl}} (\p M\times [0,1], 
g_{\p M\times [0,1]})\,.
\]
(The right hand side is in fact the relative index of $g_{\p M}$ and $g'_{\p M}$.)
However, as before, the vertical part of the metrics
$\{g(t)|_{\p M}\}_{t\in [0,1]}$ is fixed and equal  to the
metric induced by the natural one on $L$; in particular each
$\{g(t)|_{\p M}\}$ is a psc-metric; see  again Theorem \ref{thm:bundlepsc}. 
We conclude that  $g_{\p M\times [0,1]}$ is a metric of psc and so
$\alpha_{{\rm cyl}} (\p M\times [0,1], g_{\p M\times [0,1]})=0$, giving that
$\alpha_{{\rm cyl}} (M, g_M)=\alpha_{{\rm cyl}} (M, g'_M)$,
as required. This result is of course also a consequence 
of the argument below, with the bordism equal to a cylinder;
however, we think it is worthwhile to see it first and separately
as a warm-up case.

Let now $W_{\Sigma} \co M_{\Sigma}\rightsquigarrow M_{\Sigma}'$ be a bordism
between two spin pseudomanifolds with fibered $(L,G)$-singularities.
We endow $M_{\Sigma}'$ with a well-adapted wedge metric $g'$.
Recall that $W_{\Sigma} = W\cup-N(\beta W)$, where the resolution
$W$ is a manifold with corners, and its boundary $\p W$ is given a
splitting $\p W =\p^{(0)} W \cup \p^{(1)} W$, where
\[
\p^{(0)}W = M\sqcup -M', \ \ 
\p (\p^{(1)} W) = \p M \cup - \p M', \ \ \
\p(\beta W) = \beta M\sqcup -\beta M', \ \ 
\]
i.e., $\p^{(1)} W \co \p M\rightsquigarrow \p M'$ and 
$\beta W \co \beta M\rightsquigarrow \beta M'$ are usual
bordisms between closed spin manifolds. 
Also we have that the 
$(L,G)$-fiber bundle $F \co  \p^{(1)}W \to \beta W$
restricts to the $(L,G)$-fiber bundles $f\co \p M \to \beta M$ and
$f'\co \p M' \to \beta M'$, respectively. 
We must show that
$\alpha_{{\rm cyl}}  (M,g_M)= \alpha_{{\rm cyl}} (M^\prime,g'_{M'})$. 
By smoothing the corners we can assume
that the resolution $W$ is a manifold with boundary equipped with a
splitting $\p W =\p^{(0)} W \cup \p^{(1)} W$ as above. 
We can endow $\p W$ with a metric $g_{\p W}$ which is equal to the 
metric $g_{M}\sqcup (-g'_{M'})=:g^{(0)}$ on the manifold with
boundary $ M\sqcup -M' \equiv \p^{(0)} W$ 
and is equal to an extension $g^{(1)}$ of the submersion 
metric $g^{(0)}|_{\p^{(0)} W}\equiv g^{(0)}|_{\p (\p^{(1)} W)}$ 
on  $\p^{(1)} W$. As we have anticipated, 
since  $\p^{(1)} W$ is a fiber bundle with
boundary, with fiber $L$ and base $\beta W$,
we can and we shall choose the submersion metric $g^{(1)}$
to be the natural one in the vertical $L$-direction, rescaled
(by choosing the radius $R$ as in the comments following
Corollary \ref{cor:scalcurvcone})
so that the scalar curvature of the fibers is sufficiently large.
Notice that the Riemannian manifolds with boundary
$(\p^{(0)} W,g^{(0)})$ and $(\p^{(1)} W,g^{(1)})$ 
are collared near the boundary.
We extend the metric $g_{\p W}$ on $\p W$ to a collared metric $g_W$ on $W$. 
Then, by well
known bordism invariance, we have that $\alpha (\p W, g_{\p W})=0$. 
%To finish the proof, we argue as follows. We know that 
%$$\alpha _{w}(M_\Sigma,g)=\alpha _{\cyl}(M,g_M) + \alpha_{\cyl,w} (N(\beta M),g_{N (\beta M)})$$
%and similarly for $\alpha _{w}(M'_\Sigma,g')$.
On the other hand, by the gluing formula of Bunke, see \cite{Bunke},
we have that
\[
\begin{aligned}
0&=\alpha (\p W, g_{\p W})= \alpha_{{\rm cyl}} (\p^{(0)} W,g^{(0)})
+ \alpha_{{\rm cyl}}(\p^{(1)} W, g^{(1)})\\
&=
\alpha_{{\rm cyl}} (M,g_M)-\alpha_{{\rm cyl}} (M',g_{M'})
+ \alpha_{{\rm cyl}}(\p^{(1)} W, g^{(1)})\,.
\end{aligned}
\]
But $(\p^{(1)} W,g^{(1)})$ is a  Riemannian manifold with
boundary with a psc-metric.
This means that  $\alpha_{\cyl}(\p^{(1)} W,g^{(1)})=0$
and so
\[
0= \alpha _{{\rm cyl}} (M,g_M)-\alpha_{{\rm cyl}} (M^\prime,g'_{M'}),
\]
as required.

Finally, the homomorphism $\alpha_{\beta M}$ associates to
$[M_\Sigma]$ the $\alpha$-invariant of $\beta M$. This is
well-defined because it is the composition of the group
homomorphism
\[
\beta\co \Omega^{\spin,(L,G)\fb}_* \rightarrow \Omega_{*-\ell-1}^{\spin}
\]
with the well-known $\alpha$-homomorphism.
\end{proof}

\noindent
Refinements of these results will be given in  \cite{BB-PP-JR}.

\begin{remark}\label{remark-stolz' R groups}
The cylindrical class associated to a spin manifold with boundary, endowed 
with a psc metric on the boundary, together with its
relationship with bordism, has been also considered in previous work
related to Stolz' $R$-groups; see \cite{Bunke,MR1885126}.
More recent results are given in  \cite{MR3286895},
where the whole Stolz' surgery sequence is mapped to a
suitable exact $K$-theory sequence via
index theory.
\end{remark}

\section{Existence theorems}
\label{sec:existence}
\subsection{Existence when $L$ is a spin psc-$G$-boundary}
In this subsection we deal with a special case of the existence
problem for well adapted {\psc} metrics, which covers the cases where
$L=S^n$ ($n\ge 2$) or $L=G$.  This is already a large class of
situations.  Namely, we assume that our link manifold $L$ is a spin
psc-$G$-boundary in the sense of Section \ref{sec:intro}, the boundary
of a manifold $\bar L$ of {\psc}, so that the metric $g_L$ on $L$
extends nicely over $\bar L$, and the $G$-action on $L$ extends to a
$G$-action on $\bar L$.  This is clearly the case when $L=S^n$ ($n\ge
2$), $G=SO(n+1)$, and we take $\bar L$ to be the upper hemisphere in
$S^{n+1}$.  This case also applies to the case of $G=L$ a simply
connected compact Lie group, as we shall now explain.
\begin{theorem}
\label{thm:classicalgps}
Let $G$ be a simple simply connected Lie group.
Then $G$ is a spin boundary and there is a spin manifold
with boundary $\bar G$ such that $\bar G$ admits a {\psc} metric
extending the bi-invariant metric on $G$ and the $G$-action on $G$
{\lp}by left translation{\rp} extends to a $G$-action on $\bar G$.
\end{theorem}
\begin{proof}
If $G = \SU(2)\cong \Spin(3)\cong \Sp(1)\cong S^3$ has rank $1$, then
view $G=S^3$ as the boundary of the upper hemisphere $D^4$ in
$S^4$.  Thus $G$ is a spin boundary and we can put on $D^4$
a metric which in polar coordinates around the north pole
is a warped product $dr^2 + f(r)^2g_G$, where $f(r)=\sin(r)$
for $0\le r\le \frac{\pi}{2} - \frac{\varepsilon}{2}$
and $f(r)=1$ for
$\frac{\pi}{2} + \frac{\varepsilon}{2}\le r \le \frac{\pi}{2} + \varepsilon$,
which gives a nice interpolation, without changing positivity of
the scalar curvature, between the usual round metric on
$S^4$ and the cylinder metric on the product of $S^3$ with an
interval. There is a $G$-action on $D^4$ extending the $G$-action on
$G$ itself if we think of $G$ as the unit quaternions and $D^4$ as the unit
disk in $\bH$.

For $G$ of higher rank, $G$ contains a copy of $\SU(2)$ such that
the inclusion $\SU(2)\to G$ is an isomorphism on $\pi_3$.
Thus we get a fibration $\SU(2)\to G \to Y$, where $Y=G/\SU(2)$.
Replacing $\SU(2)$ in this fibration with $D^4$ with the above
metric gives a $\bar G$ with boundary $G$.

Now we need to show that $\bar G$
carries a $G$-action extending the action of $G$ on itself.
This can be shown as follows.  Note that $\bar G$ as we just
defined it is a quotient of $G\times [0, 1]$ with $G\times \{0\}$
collapsed to $Y$.  More precisely,
\[
\begin{aligned}
\bar G &= \{(g,y,t): g\in G, \,y\in Y,\, t\in [0,1],\,g\mapsto y\}/\!\sim,\\
&\text{where }(g,y,0)\sim(g',y,0)
.
\end{aligned}
\]
Note that, as required, the fiber of $\bar G$ over $y\in Y$ is just the
cone on the fiber of $G$ over $y$.  The space $\bar G$ clearly
carries a left $G$-action via
$g_1\cdot [(g,y,t)] = [(g_1g, g_1\cdot y, t)]$,
and this action extends the left $G$-action on $G$.
\end{proof}
\begin{remark}
\label{rem:morecasesGbdy}
It is easy to modify the proof to apply to a simply connected
compact Lie group that is semisimple but not simple. We leave
details to the reader.

The proof of Theorem \ref{thm:classicalgps} also applies for any
$L$ which comes with a $G$-equivariant spherical fibration
$S^k\to L\to Y$, and also to some other similar situations
which we won't outline here for lack of compelling applications.
Since $\bH\bP^1\cong S^4$, this covers the
case of examples such as the quaternionic flag manifold
$\Sp(n)/\Sp(1)^n$, $n\ge 3$, since this fibers as
\[
S^4=\bH\bP^1 \to \Sp(n)/\Sp(1)^n \to \Sp(n)/(\Sp(2)\times
\Sp(1)^{n-2})\,.
\]
  
In the literature one can find a simpler but less explicit result
than Theorem \ref{thm:classicalgps}, namely that $G$
is a spin boundary. The proof is that $G$ is parallelizable,
but the image of the natural map
$\Omega^\fr\to \Omega^\spin$ is detected by the $\alpha$-invariant
\cite[Corollary 2.7]{MR0219077}, and $G$ has a {\psc} metric,
so $G$ is trivial in $\Omega^\spin$.  But for our purposes we
need to keep track of the metric and the $G$-action as well.

Still another case where $L$ is a spin psc boundary is the
case $L=\bC\bP^{2n+1}$ of an odd complex projective space.
This can be viewed as the space of complex lines in $\bC^{2n+2} = \bH^{n+1}$,
so it fibers over $\bH\bP^n$ with fiber the space of complex lines
in a quaternionic line, or $\bC\bP^1=S^2$.  Filling in the $S^2$
with a disk shows that the standard metric on $L=\bC\bP^{2n+1}$
extends over an explicit spin boundary $\bar L$ with {\psc}.  (We found this
proof in \cite{boundaries}.)  Since
$\bH\bP^n = \Sp(n+1)/(\Sp(1)\times \Sp(n))$, we can write
$\bC\bP^{2n+1}$ as $\Sp(n+1)\times_{\Sp(1)\times \Sp(n)} S^2$,
where $\Sp(1)$ acts transitively on $S^2$ and $\Sp(n)$ acts trivially
on it, and then write
$\bar L$ as $\Sp(n+1)\times_{\Sp(1)\times \Sp(n)} D^3$.  We are not
sure if there is a choice for $\bar L$ bounding $\bC\bP^{2n+1}$ and
carrying an $\SU(2n+2)$-action, as $\Sp(n+1)$ is a smaller group than
$\SU(2n+2)$. However, taking $G=\Sp(n+1)$ is still good enough to
apply Theorem \ref{thm:Lbdy} in this context.

The case of even complex projective spaces is totally different;
these are not spin and do not bound even as non-oriented manifolds
since they have odd Euler characteristic. 
\end{remark}
\begin{theorem}
\label{thm:Lbdy}
Let $M_\Sigma\equiv (M, \partial M\to\beta M)$ be a closed $(L,G)$-singular
spin manifold.  Assume that $M$, $\beta M$, and $G$ are all
simply connected, that $n-\ell\ge 6$, and suppose that $L$ is a spin
boundary, say $L=\partial\bar L$, with the standard metric $g_L$ on $L$
extending to a {\psc} metric on $\bar L$, and with the $G$-action on
$L$ extending to a $G$-action on $\bar L$.  Assume that the two obstructions
$\alpha(\beta M)\in KO_{n-\ell-1}$ and $\alpha_{{\rm cyl}}(M)\in KO_n$
both vanish.  Then $M_\Sigma$ admits a well-adapted metric of {\psc}.
\end{theorem}
\begin{proof}
We use the bordism exact sequence \eqref{eq:5} as well as the
Bordism Theorem, Theorem \ref{thm:4-1}.

First observe that since $L$ is a spin $G$-boundary, the transfer map
$\Omega_{n-\ell-1}^\spin(BG)\to \Omega_{n-1}^\spin$ vanishes identically.  Indeed,
given any $(L,G)$-fiber bundle $\varphi\co X\to B$, since the $G$-action
on $L$ extends to a $G$-action on $\bar L$, $X$ is the spin
boundary of another fiber bundle over $B$ which is
the result of replacing each fiber $L$ with $\bar L$.
So the long exact sequence of bordism groups becomes a short exact
sequence
\begin{equation}
\label{eq:bordismSES}
0\to \Omega_n^\spin \xrightarrow{\iota} \Omega_n^{\spin,(L,G)\fb}\xrightarrow{\beta}
\Omega_{n-\ell-1}^\spin(BG)\to 0.
\end{equation}

Now suppose that $(M, \partial M\to\beta M)$ is as in the theorem.
We will construct another $(L,G)$-singular spin manifold in the
same bordism class with a well-adapted
metric of {\psc}.  Then
$M_\Sigma$ will admit a well-adapted metric of {\psc} by Theorem
\ref{thm:4-1}. By assumption, $\alpha(\beta M)=0$, $\beta M$ is simply
connected, and $\dim \beta M \ge 5$. So by Stolz's Theorem,
\cite{MR1189863}, $\beta M$ has a Riemannian metric of {\psc}.
Use a connection on the $G$-bundle over $\beta M$ associated to $\p M$ 
in order to construct a
well-adapted metric of {\psc} on the tubular neighborhood $N$ of
$\beta M$ in $M_\Sigma$.  The boundary of $N$ is an $(L,G)$-fiber
bundle over $\beta M$ with a {\psc} metric with a Riemannian
submersion to $\beta M$.  Let $M'=P\times_G \bar L$ be the
$\bar L$-bundle over $\beta M$ associated to the corresponding
principal $G$-bundle $P\to \beta M$. Then $M'$ has a bundle metric of
{\psc}, and joining $M'$ to $N$, we get an $(L,G)$-singular spin
manifold $M'_\Sigma$ with a well-adapted metric of {\psc}. Since
$M'_\Sigma$ and $M_\Sigma$ coincide near $\beta M$, by
\eqref{eq:bordismSES}, their bordism classes differ by a class in
the image of $\Omega_n^\spin$; that is, there exists a closed spin manifold 
$N$ such that 
$$
[M'_\Sigma]- [M_\Sigma]=\iota [N]\,.$$
Consider now the following diagram: 
\[
\xymatrix{\Omega_n^\spin \ar[r]^{\iota} \ar[d]^{\alpha} &  \Omega_n^{\spin,(L,G)\fb}
\ar[d]^{\alpha_{{\rm cyl}}}\ar[r]^{\beta} &
 \Omega_{n-\ell-1}^\spin(BG)\\
KO_n \ar[r]^{{\rm Id}} & KO_n 
& \;}
\]
which is obviously commutative.
By assumption  $\alpha_{{\rm cyl}}(M)=0$ in
$KO_n$, and also $\alpha_{{\rm cyl}}(M')=0$ since $M'$ has {\psc}.
Thus from the above diagram we infer that $\alpha (N)=0$.
By doing surgery we can assume that $N$ is spin-bordant to a closed
spin manifold $M''$ which is simply connected. Thus $\alpha (M'')=0$,
$M''$ is simply connected and 
$[M'_\Sigma]- [M_\Sigma]=\iota [M'']\,.$
Summarizing, $M_\Sigma$ is in the same bordism class
in $\Omega_n^{\spin,(L,G)\fb}$ as $M'_\Sigma \amalg M''$,
where $M''$ is a closed spin $n$-manifold with {\psc} and where $M'_\Sigma$
is a $(L,G)$-pseudomanifold with psc.  Now we can
apply Theorem \ref{thm:4-1} to get the conclusion.
\end{proof}

\subsection{Existence when the map to $BG$ is null-bordant}
In this subsection we consider a different case of the existence
problem, the case where $\beta M\to BG$ represents a trivial
element in $\Omega_*^\spin(BG)$.  This case will also be easy to
deal with, using the bordism exact sequence \eqref{eq:5} and the
Bordism Theorem, Theorem \ref{thm:4-1}.
\begin{theorem}
\label{thm:trivialbundle}
Let $M_\Sigma\equiv (M, \partial M\to\beta M)$ be a closed $(L,G)$-singular
spin manifold $M_\Sigma$. Assume that $M$, $\beta M$, and $G$ are all
simply connected and that $n-\ell\ge 6$. We make no additional
assumptions on $G$ and $L$, but we assume that
the class of $\beta M\to BG$ represents $0$ in $\Omega_{n-\ell-1}^\spin(BG)$
and that $\alpha_{{\rm cyl}}(M)\in KO_n$ vanishes.
Then $M_\Sigma$ admits a well-adapted metric of {\psc}.
\end{theorem}
\begin{remark}
The assumption of this theorem is admittedly somewhat special, but is
satisfied more often than one might expect.  First of all, vanishing
of the bordism class of $\beta M\to BG$ is weaker than assuming both
that $\beta M$ is a spin boundary and that the $G$-bundle over $\beta M$ is
trivial (i.e., $\beta M\to BG$ is null-homotopic).  For example,
if $G=\SU(2)$, then $BG=\bH\bP^\infty$ and all torsion in
$\Omega_*^\spin(\bH\bP^\infty)$ is $2$-primary.  But if $\beta M$ is a sphere,
the homotopy class of $\beta M\to \bH\bP^\infty$ lies in
$\pi_{n-\ell-1}(\bH\bP^\infty)=\pi_{n-\ell-2}(S^3)$.  It is a simple well-known
fact \cite[Corollary 1.2.4]{MR860042}
that the homotopy groups of $S^3$ contain torsion of order $p$
for any prime $p$. However, none of the odd torsion shows up in
$\Omega_*^\spin(\bH\bP^\infty)$.

Secondly, there are some cases where the main assumption of the
theorem is automatic, namely cases where the kernel of the
$(L,G)$-transfer map $\Omega_{n-\ell-1}^\spin(BG)\to \Omega_{n-1}^\spin$
is trivial.  Just as an example, if $G$ is a
symplectic group and $L$ is a quaternionic projective space, then
$\Omega_*^\spin(BG)$ is $2$-primary torsion except in dimensions
divisible by $4$.  So if $n-\ell$ is not $1\mod 4$, then the class of
$\beta M\to BG$ is at most $2$-primary torsion in the bordism group,
and some finite multiple of it satisfies the hypotheses of the theorem.
\end{remark}  
\begin{proof}[Proof of Theorem \ref{thm:trivialbundle}]
By assumption, there is a spin manifold $W$ with boundary
and a principal $G$-bundle over $W$ with $\partial W = \beta M$
and the bundle on $W$ extending the principal $G$-bundle over
$\beta M$. Let $M'$ be the associated $L$-bundle over $W$.
Then $\partial M' = \partial M$.  Choose a metric of {\psc} on $W$
restricting to a product metric of {\psc} in a neighborhood of
$\beta M$.  Using this metric and the bundle structure over $\beta M$,
we get a manifold $M'_\Sigma$ with $(L,G)$-fibered singularities
with a well-adapted metric of {\psc}.  By the bordism exact sequence
\eqref{eq:5}, the difference between the bordism class of $M_\Sigma$
and the bordism class of $M'_\Sigma$ lies in the image of
$\Omega_n^\spin$. So $M_\Sigma$ is spin bordant to $M'_\Sigma \amalg M''$,
where $M''$ is a closed spin manifold.  By additivity of the
$\alpha$-invariant and the assumption that $\alpha_{{\rm cyl}}(M)\in
KO_n$ vanishes, $\alpha(M'')=0$. So the result follows from the
Bordism Theorem, Theorem \ref{thm:4-1}.
\end{proof}

We proceed to give another application of Theorem
\ref{thm:trivialbundle}. Suppose we look only at manifolds  with
Baas-Sullivan singularities, i.e., we require that
$\p M = \beta M \times L$, and suppose $L=\bH\bP^{2k}$, $k\ge 1$,
$G=\Sp(2k+1)$.  (This is one of the key examples where $L$ is
\emph{not} a spin boundary, so that Theorem \ref{thm:Lbdy} doesn't apply.)
\begin{lemma}
\label{lem:HP2}
The class of $\bH\bP^{2k}$, $k\ge 1$, is not a zero-divisor in the
spin bordism ring $\Omega_*^\spin$.
\end{lemma}
\begin{proof}
Note that $\bH\bP^{2k}$ has nonzero signature and odd Euler
characteristic, so it represents nontrivial elements of
$\Omega_*^\spin\otimes \bQ$ and of
$\fN_*$, the non-oriented bordism ring,
which are polynomial rings (over $\bQ$ and
$\bF_2$, respectively).  So if $x\in \Omega_*^\spin$ is nonzero either in
$\Omega_*^\spin\otimes \bQ$ or in $\fN_*$, its product with
$\bH\bP^{2k}$ can't be a spin boundary, and hence the image of
$\bH\bP^{2k}$ in $\Omega_*$ is not a zero-divisor. So for
$\bH\bP^{2k}$ to be a zero-divisor in $\Omega_*^\spin$, it would have
to annihilate a non-zero element $x$ in the kernel of the forgetful
map $\Omega_*^\spin\to \Omega_*$.  Now apply \cite[Corollaries 2.3 and
  2.6]{MR0219077}. This element $x$ would have to live in dimension
$1$ or $2$ mod $8$ and be of the form $(\hbox{torsion-free
  element})\times \eta^j$, $j=1$ or $2$.  (Here $\eta$ is the usual
generator of $\Omega^\spin_1$.)  But multiplying such an element by
$\bH\bP^{2k}$ would give another element of the same form (in
dimension $8k$ higher) which would be non-zero again.  So
$\bH\bP^{2k}$ cannot be a zero-divisor.
\end{proof}
\begin{theorem}
\label{thm:BaasHP2}
Let $M_\Sigma\equiv (M^n, \partial M\to\beta M)$ be a closed $(L,G)$-singular
spin manifold $M_\Sigma$, with $L=\bH\bP^{2k}$ and $G=\Sp(2k+1)$,
$n\ge 1$.  
Assume that $\p M = \beta M \times L$, i.e., the $L$-bundle over
$\beta M$ is trivial, or in other words that the singularities
are of Baas-Sullivan type.  Then if $M$ and $\beta M$ are both
simply connected and $n-8k\ge 6$, $(M,\p M)$ has an adapted metric
of {\psc} if and only if the $\alpha$-invariants
$\alpha(\beta M) \in KO_{n-8k-1}$ and $\alpha_{{\rm cyl}}(M) \in KO_n$
both vanish.
\end{theorem}
\begin{proof}
In  the case where the $L$-bundle over $\beta M$ is trivial, the
$(L,G)$-transfer map $\Omega_{n-\ell-1}^\spin \to \Omega_{n-1}^\spin$ is
just multiplication by the class of $L$ in $\Omega_\ell^\spin$.
When $L=\bH\bP^{2k}$, by Lemma
\ref{thm:BaasHP2}, this map is injective, and thus $[\beta M]=0$ in
$\Omega_{n-\ell-1}^\spin$. So we can apply Theorem
\ref{thm:trivialbundle}.
\end{proof}  
\begin{remark}
Note that Lemma \ref{lem:HP2} fails for odd quaternionic projective spaces,
since these annihilate torsion classes in the kernel of
the forgetful map $\Omega_*^\spin\to \Omega_*$.  Nevertheless, we expect
that the case of Baas-Sullivan singularities with $L=\bH\bP^{2k+1}$
is treatable, but will require a more complicated argument.  We leave this
to future work.
\end{remark}  
\section{Preview of the non-simply connected case}
\label{sec:preview}

We conclude by mentioning some ``unfinished business'' that
will be treated in a continuation of this paper \cite{BB-PP-JR}.
We start by extending the obstruction theory to the cases
where $M_\Sigma$ and/or $\beta M$ are not simply connected.
As in the theory of psc on general closed manifolds
(see \cite{MR2408269} for a survey), this involves
obstructions in the $KO$-theory of the group $C^*$-algebras
of the relevant fundamental groups.  Then we generalize the
Surgery Theorem and Bordism Theorem (Theorems \ref{thm:surg}
and \ref{thm:4-1}) to this situation.  As a result we are
able to generalize the Existence Theorem (Theorem \ref{thm:pscsufficiency})
to the case where the relevant groups satisfy the
Gromov-Lawson-Rosenberg conjecture.  If the groups
are in the class where the Baum-Connes assembly map
is injective, then
by a theorem of Stolz \cite[\S3]{MR1937026}, we can at least prove a
``stable'' existence theorem in the sense of \cite{MR1321004}.

Other problems to be discussed in \cite{BB-PP-JR} involve
the topology of the space of well-adapted psc-metrics if this
space is non-empty.  In some cases where $M_\Sigma$ is not
simply connected, rho-invariants on manifolds with $L$-fibered singularities
can be used to show
that this space has infinitely many components.
We will also see that the topology of the space of 
well-adapted psc-metrics on $M_\Sigma$ is at least as complicated
as that of the space of psc-metrics on $\beta M$.

\bibliographystyle{amsplain}

\bibliography{PSCFiberedSingGen}

%\affiliationone{Boris Botvinnik\\
%  Department of Mathematics\\
%  University of Oregon\\
%  Eugene OR 97403-1222, USA\\
%\email{botvinn@uoregon.edu}}
%\affiliationtwo{Paolo Piazza\\
%Dipartimento di Matematica\\ Universit\`a di Roma La
%Sapienza\\ Piazzale Aldo Moro\\ 00185 Roma, Italy\\
%\email{piazza@mat.uniroma1.it}}
%\affiliationthree{Jonathan  Rosenberg\\
%  Department of Mathematics\\ University of
%  Maryland\\ College Park, MD 20742-4015, USA\\
%  \email{jmr@math.umd.edu}}
\end{document}